\documentclass[12pt,a4paper]{article}
\usepackage{amssymb,latexsym,amsmath}
\usepackage{float}              

\usepackage{epsfig}
\usepackage{subfigure}
\usepackage{subfig}
\usepackage{graphicx}
\usepackage{xypic}
\usepackage{xcolor}
\usepackage[all,knot,arc,import,poly]{xy}
\usepackage{caption}

\input epsf

\newtheorem{prop}{Proposition}[section]
\newtheorem{theo}[prop]{Theorem}

\newtheorem{ex}[prop]{Example}
\newtheorem{defn}[prop]{Definition}
\newtheorem{rem}[prop]{Remark}

\newtheorem{lem}[prop]{Lemma}

\newenvironment{proof}
 {\begin{trivlist} \item[\hskip \labelsep {\bf Proof}\hspace*{3 mm}]}
 {\hfill$\Box$\end{trivlist}}

\begin{document}

\title{Singular 3-manifolds in $\mathbb{R}^5$}
\author{Pedro Benedini Riul\footnote{Supported by FAPESP Grant 2019/00194-6}, Maria Aparecida Soares Ruas\footnote{Supported by CNPq grant
 402181/2015-8 and FAPESP grant 2014/00304-2 }\\ and  Andrea de Jesus Sacramento\footnote{Supported by CNPq grant
 150469/2017-9}}

\date{}

\maketitle
\begin{abstract}
We study 3-manifolds in $\mathbb{R}^5$ with corank $1$ singularities. At the singular point we define the curvature
locus using the first and second fundamental forms, which contains all the local second order geometrical information about the manifold.
\end{abstract}

\section{Introduction}
The aim of this paper is to study the  curvature locus of a 3-manifold in $\mathbb{R}^5$ with corank 1 singularities. Our investigation is motivated by
\cite{luciana} and \cite{Pedro/Cidinha/Raul} in which the authors study, respectively, singular surfaces in $\mathbb{R}^3$ and $\mathbb{R}^4$ with corank 1
singularities. More specifically, they both define the curvature parabola, a plane curve in the normal space that plays a resembling role as the curvature
ellipse for regular surfaces in Euclidean spaces, using the first and second fundamental forms of the surface at the singular point.
The regular case has been known since a long time ago 
(see \cite{karl}), and the curvature ellipse has proven to be a useful tool in the study of geometrical properties from both, the local and global viewpoint
\cite{fuster}, \cite{little} and \cite{mochida}.

For smooth $3$-manifolds immersed in $\mathbb{R}^n$, $n\geq5$, the geometry of the second fundamental form at a point is determined by the geometry of a set called the curvature locus. In \cite{carmen,Carmen3var2} the authors study the behaviour of this curvature locus, describing their possible topological types.
As the previous cases, we define the curvature locus as the image of the unit tangent vectors via the second fundamental form.


Let $M\subset\mathbb{R}^5$ be a $3$-manifold with a singularity of corank 1. This means that $M$ is the image of a smooth map $g:\widetilde{M}\rightarrow \mathbb{R}^5$ from a smooth regular $3$-manifold $\widetilde{M}$ whose differential map has rank $\geq1$ at any point. Hence, the tangent space $T{_p}M$ at a singular point $p$ degenerates to a 2-space and $N_{p}M$ is the $3$-space of directions orthogonal to $T_pM$ in $\mathbb{R}^5$.

Our paper is organized as follows: In section \ref{sec:fundamental form} we define the first and second fundamental forms of $M$ at $p$ and study their properties.
In Section \ref{sec:locus} using these fundamental forms, we define the curvature locus, $\Delta_p$, as a subset of $N_pM$ (see Definition \ref{def:curvature locus}). The curvature locus can be a substantial surface (that is, a surface with Gaussian curvature non identically zero), a planar region, a half-line, a line, a point, etc.

Given a corank $1$ surface in $\mathbb{R}^n$, $n=3,4$, the topological type of its curvature parabola can distinguish the orbits of the $\mathcal{A}^2$-classification ($\mathcal{A}^2$ is the space of 2-jets of diffeomorphisms in the source and target) of corank $1$ map germs $f:(\mathbb{R}^2,0)\rightarrow(\mathbb{R}^n,0)$, as shown in \cite{Pedro/Cidinha/Raul,luciana}.
Here, we also give a partition in six orbits of the set of all corank 1 map germs $f:(\mathbb{R}^3,0)\rightarrow(\mathbb{R}^5,0)$ according to their 2-jets under the action of $\mathcal{A}^2$ (see Proposition \ref{prop:orbitas}). However, in this case each orbit has
more than one possibility for the curvature locus and it can not be used to distinguish between the singularities (see Remark \ref{rem:obs orbitas}).
Even though a general result is not true, we proved a partial result, for a special case of corank $1$ $3$-manifold (see Theorem \ref{teo-partial}).

Furthermore, we show that two corank $1$ $2$-jets $(\mathbb{R}^3,0)\rightarrow(\mathbb{R}^5,0)$ are equivalent under the action of the subgroup $\mathcal{R}^2\times\mathcal{O}(5)$ of 2-jets of diffeomorphisms in the source and linear isometries of $\mathbb{R}^5$ if and only if there exists an isometry between the normal spaces preserving the respective curvature locus (Theorem \ref{teo:isometria}, whose proof is presented
in the Appendix at the end of paper). In  Section \ref{sec:geometria locus} we review some definitions and results from \cite{carmen,Carmen3var2}, on the second order geometry of regular $3$-manifolds in $\mathbb{R}^6$.

In \cite{wall/duplesis} the authors present a $\mathcal{G}=GL(3)\times GL(3)$-classification of nets of quadrics. To each net we can naturally associate a regular $3$-manifold in $\mathbb{R}^6$ whose second fundamental form is the net. A question that may arise is whether given two $\mathcal{G}$-equivalent nets,
the loci of curvatures of the two associated manifolds are topologically equivalent. The answer to this question is no and we give an example.

Finally, when projecting a regular $3$-manifold immersed in $\mathbb{R}^6$ in a tangent direction, we obtain a corank $1$ $3$-manifold in $\mathbb{R}^5$.
The last section shows that the curvature locus at a singular point
of a 3-manifold is the blow up of the curvature locus of the regular $3$-manifold from which the singular was projected, similarly to what was done in \cite{Pedro/Raul} for regular surfaces in $\mathbb{R}^4$ and singular surfaces in $\mathbb{R}^3$. Some examples are also presented to illustrate.


Aknowledgements: the authors would like to thank R. Oset Sinha for useful comments and constant encouragement.

\section{The second fundamental form  at a corank 1 singularity}\label{sec:fundamental form}

Let $M$ be a 3-manifold with a singularity of corank 1 at $p\in M$. We assume that $M$ is the image of a $C^{\infty}$ map $g:\widetilde{M}\rightarrow\mathbb{R}^5$, where $\widetilde{M}$ is a smooth regular 3-manifold and $q\in \widetilde{M}$ is a
singular corank 1 point of $g$ such that $g(q)=p$.

Given a coordinate system $\phi:U\rightarrow \mathbb{R}^3$ defined on some open neighbourhood $U$ of $q$ in $\widetilde{M}$,
we say that $f=g\circ\phi^{-1}$ is a local parametrisation of $M$ at $p$. See the diagram below.

$$
\xymatrix{
\mathbb{R}^{3}\ar@/_0.7cm/[rr]^-{f} & U\subset\widetilde{M}\ar[r]^-{g}\ar[l]_-{\phi} & M\subset\mathbb{R}^{5}
}
$$

We define the \emph{tangent space} to $M$ at $p$ as $T{_p}M=Im(dg{_q})$, where $dg_{q}:T_{q}\widetilde{M}\rightarrow T_{q}\mathbb{R}^5$ is the differential map of $g$ at $q$. Once $q\in\widetilde{M}$ is a corank 1 point of the map $g$, $\dim(T_pM)=2$. The \emph{normal space} $N_{p}M$ at $p$, is the 3-dimensional vector space such that $T_{p}\mathbb{R}^5=T_{p}M\oplus N_{p}M$. We denote the corresponding orthogonal projections by:
$$\begin{array}{cccc}
  \top: & T_{p}\mathbb{R}^5 & \rightarrow & T_{p}M \\
   & w & \rightarrow& w^{\top}
\end{array}\hspace{1.0cm} \begin{array}{cccc}
  \bot : &T_{p}\mathbb{R}^5 & \rightarrow & N_{p}M \\
   & w & \rightarrow& w^{\bot}
\end{array} $$

The Euclidian metric of $\mathbb{R}^5$ induces the \emph{first fundamental form} $\mathcal{I}:T_{q}\widetilde{M}\times T_{q}\widetilde{M}\rightarrow \mathbb{R}$:
$$\mathcal{I}(\textbf{u},\textbf{v})=\langle dg_{q}(\textbf{u}),dg_q(\textbf{v})\rangle,\,\,\,\,\mbox{for\,\,all}\,\,\textbf{u},\textbf{v}\in T_{q}\widetilde{M} .$$

Observe that $\mathcal{I}$ is not a Riemanniam metric on $T_{q}\widetilde{M}$. However $I$ is a pseudometric. Taking a local parametrisation $f=g\circ\phi^{-1}$ of $M$ at $p$ and $\{\partial x,\partial y,\partial z\}$ a basis of $T_{q}\widetilde{M}$, the coefficients of the first fundamental form with respect to $\phi$ are:
$$\begin{array}{ccc}
  E(q)= & \mathcal{I}(\partial x,\partial x)= & \langle f_x,f_x\rangle(\phi(q)) \\
  F(q)= & \mathcal{I}(\partial x,\partial y)= & \langle f_x,f_y\rangle(\phi(q))\\
  G(q)= & \mathcal{I}(\partial y,\partial y)= & \langle f_y,f_y\rangle(\phi(q)) \\
  H(q)= & \mathcal{I}(\partial z,\partial z)= & \langle f_z,f_z\rangle(\phi(q)) \\
  I(q)= & \mathcal{I}(\partial x,\partial z)= & \langle f_x,f_z\rangle(\phi(q)) \\
  J(q)= & \mathcal{I}(\partial y,\partial z)= & \langle f_y,f_z\rangle(\phi(q))
\end{array}$$

Notice that if $\textbf{u}=a\partial x+b\partial y + c\partial z$ then $$\mathcal{I}(\textbf{u},\textbf{u})=a^2E(q)+2abF(q)+b^2G(q)+c^2H(q)+2acI(q)+2bcJ(q).$$

We can also define the \emph{second fundamental form} $\mathcal{II}:T_{q}\widetilde{M}\times T_{q}\widetilde{M}\rightarrow N_{p}M$ of $M$ at $p$. Given local coordinates of $M$ at $p$ as before, we define

$$\begin{array}{ccc}
    \mathcal{II}(\partial x,\partial x)= f_{xx}^{\bot}(\phi(q)),\,\,&\mathcal{II}(\partial x,\partial y)= f_{xy}^{\bot}(\phi(q)),\,\,&   \mathcal{II}(\partial y,\partial y)= f_{yy}^{\bot}(\phi(q)),  \\
    \mathcal{II}(\partial z,\partial z)= f_{zz}^{\bot}(\phi(q)),\,\,&\mathcal{II}(\partial x,\partial z)= f_{xz}^{\bot}(\phi(q)),\,\,&\mathcal{II}(\partial y,\partial z)= f_{yz}^{\bot}(\phi(q)),
  \end{array}
$$
and we extend $\mathcal{II}$ to $T_{q}\widetilde{M}\times T_{q}\widetilde{M}$ in a unique way as a symmetric bilinear map.
Let $\textbf{u}=a\partial x+b\partial y + c\partial z\in T_q\widetilde{M}$ be a tangent vector. Then,  $$\mathcal{II}(\textbf{u},\textbf{u})=a^2f_{xx}^{\bot}(\phi(q))+2abf_{xy}^{\bot}(\phi(q))+b^2f_{yy}^{\bot}(\phi(q))+c^2f_{zz}^{\bot}(\phi(q))+2acf_{xz}^{\bot}(\phi(q))+2bcf_{yz}^{\bot}(\phi(q)).$$

\begin{lem}\label{lem:coordinates}
The definition of the second fundamental form does not depend on the choice of local coordinates on $\widetilde{M}$.
\end{lem}
\begin{proof}
To prove this lemma consider another coordinate system of $\widetilde{M}$ at $q$  and apply the chain rule (see, for example, Lemma 2.1 in \cite{luciana}).
\end{proof}

For each normal vector $v\in N_{p}M$, we can consider the form $\mathcal{II}_v:T_{q}\widetilde{M}\times T_{q}\widetilde{M}\rightarrow \mathbb{R}$ which we call the \emph{second fundamental form of $M$ at $p$ along $v$}, given by $$\mathcal{II}_v(\textbf{u},\textbf{v})=\langle \mathcal{II}(\textbf{u},\textbf{v}),v\rangle.$$

The coefficients of $\mathcal{II}_v$ in terms of local coordinates $(x,y,z)$ are:
$$\begin{array}{ccc}
    l_v(q)=\langle f_{xx}^{\bot},v\rangle(\phi(q)),\,\, & m_v(q)= \langle f_{xy}^{\bot},v\rangle(\phi(q)),\,\,  & n_v(q)= \langle f_{yy}^{\bot},v\rangle(\phi(q)),  \\
    p_v(q)= \langle f_{zz}^{\bot},v\rangle(\phi(q)),\,\, & q_v(q)= \langle f_{xz}^{\bot},v\rangle(\phi(q)),\,\,  & r_v(q)= \langle f_{yz}^{\bot},v\rangle(\phi(q)).
  \end{array}
$$

Given $\textbf{u}=a\partial x+b\partial y + c\partial z\in T_q\widetilde{M}$ then $$\mathcal{II}_v(\textbf{u},\textbf{u})=a^2 l_v(q)+2abm_v(q)+b^2n_v(q)+c^2p_v(q)+2acq_v(q)+2bcr_v(q).$$

For a fixed orthonormal frame $\{v_1,v_2,v_3\}$ of $N_{p}M$, the quadratic form associated to the second fundamental form can be written as
\begin{equation}\label{eq:secondform}
  \mathcal{II}(\textbf{u},\textbf{u})= \mathcal{II}_{v_1}(\textbf{u},\textbf{u})v_1+\mathcal{II}_{v_2}(\textbf{u},\textbf{u})v_2 +\mathcal{II}_{v_3}(\textbf{u},\textbf{u})v_3,
\end{equation}
where $\mathcal{II}_{v_i}(\textbf{u},\textbf{u})=a^2 l_{v_i}+2abm_{v_i}+b^2n_{v_i}+c^2p_{v_i}+2acq_{v_i}+2bcr_{v_i}$, for $i=1,2,3$
and the above coefficients calculated at $q$. Furthermore, in terms of the chosen frame, the second fundamental form can be represented by the following $3\times6$ matrix of coefficients: $$\left(
                                                         \begin{array}{cccccc}
                                                           l_{v_1} & m_{v_1}& n_{v_1} & p_{v_1} & q_{v_1} & r_{v_1} \\
                                                           l_{v_2} & m_{v_2} & n_{v_2} & p_{v_2} & q_{v_2} & r_{v_2}\\
                                                           l_{v_3} & m_{v_3} & n_{v_3} & p_{v_3} & q_{v_3} & r_{v_3}
                                                         \end{array}
                                                       \right).
$$

\section{The curvature locus}\label{sec:locus}

Inpired by the definition of the curvature parabola for corank 1 surfaces in $\mathbb{R}^n$, $n=3,4$ we shall give a natural definition for the curvature locus
of corank 1 3-manifolds in $\mathbb{R}^5$.

\begin{defn}\label{def:curvature locus}
Let $C_q$ be the subset of unit vectors of $T_q\widetilde{M}$ and let $\eta:C_q\rightarrow N_pM$ be the map given by $\eta(\textbf{u})=\mathcal{II}(\textbf{u},\textbf{u}).$ We define the \emph{curvature locus} of $M$ at $p$, which we shall denote by $\Delta_p$, as the image of this map, that is, $\Delta_p=\eta(C_q)$.
\end{defn}

From the previous definition, $\Delta_p=\{\mathcal{II}(\textbf{u},\textbf{u})\mid \mathcal{I}(\textbf{u},\textbf{u})^{1/2}=1\}$ and from Lemma \ref{lem:coordinates}, the definition of the curvature locus does not depend on the choice of the local coordinates of $\widetilde{M}$.

\begin{ex}\label{ex:paraboloide}
Let $M$ be the image of $\widetilde{M}=\mathbb{R}^3$ by the map $g(x,y,z)=(x,y,xz,yz,z^2)$. So $M$ is the $3$-dimensional crosscap.
Taking coordinates $(u,v,w,u',v')\in \mathbb{R}^5$, $p=(0,0,0,0,0)$ and $q=(0,0,0)$, the tangent space $T_pM$ to $M$ at $p$ is the $(u,v)$-space and so, the normal space $N_pM$ is the $(w,u',v')$-space. Then, for any $\textbf{u}=a\partial x+b\partial y + c\partial z\in T_q\widetilde{M}$, we have $E=G=1$, $F=H=I=J=0$. The fundamental forms are given by $\mathcal{I}(\textbf{u},\textbf{u})=a^2+b^2$ and $\mathcal{II}(\textbf{u},\textbf{u})=(0,0,2ac,2bc,2c^2)$. Hence, the subset of unit tangent vectors is the cylinder $C_p=\{(a,b,c)\in\mathbb{R}^3\mid a^2+b^2=1\}$ and the curvature locus $$\Delta_p=\{(0,0,2ac,2bc,2c^2)\mid a^2+b^2=1\}$$ is a paraboloid.
\end{ex}

\begin{ex} Other examples of curvature loci for corank 1 3-manifolds in $\mathbb{R}^5$ are the following:
\begin{itemize}
\item[(i)] Let $g(x,y,z)=(x,y,z^2,xz,0)$ then $$\Delta_p=\{(0,0,2c^2,2ac,0)\mid a^2+b^2=1,c\in\mathbb{R}\}$$ is a parabolic region for all $p=(0,y,0,0,0)$ with $y\in \mathbb{R}$.

\item[(ii)] Let $g(x,y,z)=(x,y,xz,yz,0)$ then for $p=(0,0,0,0,0)$ we have that $$\Delta_p=\{(0,0,2ac,2bc,0)\mid a^2+b^2=1,c\in\mathbb{R}\}$$ is a plane.

\item[(iii)] Let $g(x,y,z)=(x,y,z^2,0,0)$ then for $p=(0,0,0,0,0)$ we have $$\Delta_p=\{(0,0,2c^2,0,0)\mid a^2+b^2=1,c\in\mathbb{R}\}$$ is a half-line.

\item[(iv)] Let $g(x,y,z)=(x,y,xz,0,0)$ then for $p=(0,0,0,0,0)$ we have $$\Delta_p=\{(0,0,2ac,0,0)\mid a^2+b^2=1,c\in\mathbb{R}\}$$ is a line in $N_pM$ through the origin.

\item[(v)] Let $g(x,y,z)=(x,y,0,0,0)$ then for $p=(0,0,0,0,0)$ we have that $\Delta_p$ is just the point $(0,0,0,0,0)$.
\end{itemize}
\end{ex}

\begin{rem}\label{obs:cilindro}
Since the map $g$, used for the initial construction, has corank 1 at $q$, we can choose the coordinate system $\phi$ and make rotations in $\mathbb{R}^5$ such  that
$$f(x,y,z)=(x,y,f_1(x,y,z),f_2(x,y,z),f_3(x,y,z)),$$
with $(f_i)_x=(f_i)_y=(f_i)_z=0$ at $\phi(q)$, for $i=1,2,3$. In this case, the coefficients of the first fundamental form are given by $E=G=1$ and $F=H=I=J=0$. With these coordinates, given a unit tangent vector $\textbf{u}\in C_q$ and writing $\textbf{u}=x\partial x+y\partial y+z\partial z$, since $$x^2E(q)+2xyF(q)+y^2G(q)+z^2H(q)+2xzI(q)+2yzJ(q)=1$$
we have $x^2+y^2=1$, that is, $C_q$ is a  unit cylinder parallel to the $z$-axis. Hence, fixing an orthonormal frame $\{v_1,v_2,v_3\}$ of $N_pM$
and using (\ref{eq:secondform}), it follows that
$$(x,y,z)\mapsto\sum_{i=1}^{3}(x^2 l_{v_i}+2xym_{v_i}+y^2n_{v_i}+z^2p_{v_i}+2xzq_{v_i}+2yzr_{v_i})v_i$$
%
is a parametrisation for curvature locus $\Delta_p$ in the normal space, where $x^2+y^2=1$.
\end{rem}


The next result presents a partition of all corank 1 map germs $f:(\mathbb{R}^3,0)\rightarrow(\mathbb{R}^5,0)$ according to their 2-jet under the action of $\mathcal{A}^2$, which denotes the space of 2-jets of diffeomorphisms in source and target. We denote by $J^2(3,5)$ the subspace of 2-jets $j^2f(0)$ of map germs $f:(\mathbb{R}^3,0)\rightarrow(\mathbb{R}^5,0)$ and by $\Sigma^1J^2(3,5)$ the subset of 2-jets of corank 1.

\begin{prop}\label{prop:orbitas}
There exist six orbits in $\Sigma^1J^2(3,5)$ under the action of $\mathcal{A}^2$, which are $$(x,y,xz,yz,z^2),\,(x,y,z^2,xz,0),\,(x,y,xz,yz,0),\,(x,y,z^2,0,0),\,(x,y,xz,0,0),\,(x,y,0,0,0).$$
\end{prop}
\begin{proof}
Consider the $2$-jet

$$
\begin{array}{cl}
j^2f(0)=&(x,y,a_{20}x^2+a_{11}xy+a_{02}y^2+a_{21}z^2+a_{22}xz+a_{12}yz,b_{20}x^2+b_{11}xy+b_{02}y^2\\
&+b_{21}z^2+b_{22}xz+b_{12}yz,c_{20}x^2+c_{11}xy+c_{02}y^2+c_{21}z^2+c_{22}xz+c_{12}yz).
\end{array}
$$
Using coordinate changes in the target we can remove the terms with $x^2$, $y^2$ and $xy$ of the three coordinates of $j^2f(0)$ and thus we obtain $$j^2f(0)\sim (x,y,a_{21}z^2+a_{22}xz+a_{12}yz,b_{21}z^2+b_{22}xz+b_{12}yz,
c_{21}z^2+c_{22}xz+c_{12}yz).$$
Let
\begin{equation}\label{Det}
\alpha=\left(
                    \begin{array}{ccc}
                      a_{21} & a_{22} & a_{12} \\
                      b_{21} & b_{22} & b_{12} \\
                      c_{21} & c_{22} & c_{12} \\
                    \end{array}
                  \right)\ \mbox{and}\ D=\det(\alpha).
\end{equation}
Suppose that $D\neq0$. Thus $(a_{21},b_{21},c_{21})\neq 0$ and
without loss of generality we can suppose $c_{21}\neq0$. Making the change of coordinates in the target $$\widetilde{X}=X,\,\,\,\widetilde{Y}=Y,\,\,\,\widetilde{Z}=Z-\frac{a_{21}}{c_{21}}T,\,\,\,\widetilde{W}=W-\frac{b_{21}}{c_{21}}T\,\,\,\mbox{and}\,\,\,\widetilde{T}=T,$$

\noindent we have

$
\begin{aligned}
 j^2f(0)&\sim\Big(x,y,\dfrac{(a_{22}c_{21}-a_{21}c_{22})}{c_{21}}xz+\dfrac{(a_{12}c_{21}-a_{21}c_{12})}{c_{21}}yz,\dfrac{(b_{22}c_{21}-b_{21}c_{22})}{c_{21}}xz\\
 & + \dfrac{(b_{12}c_{21}-b_{21}c_{12})}{c_{21}}yz,c_{21}z^2+c_{22}xz+c_{12}yz\Big).
\end{aligned}.
$

Since $D\neq0$, we can also assume $(b_{12}c_{21}-b_{21}c_{12})\neq0$, $(a_{22}c_{21}-a_{21}c_{22})\neq0$
and making the change  of coordinates $$\overline{X}=X,\,\,\,\overline{Y}=Y,\,\,\,\overline{Z}=Z-\frac{(a_{12}c_{21}-a_{21}c_{12})}{b_{12}c_{21}-b_{21}c_{12}}W,\,\,\,
\overline{W}=W-\frac{(b_{22}c_{21}-b_{21}c_{22})}{a_{22}c_{21}-a_{21}c_{22}}Z\,\,\,\mbox{and}\,\,\,\overline{T}=T,$$ we obtain

$$j^2f(0)\sim\left(x,y, \frac{D}{b_{12}c_{21}-b_{21}c_{12}}xz,\frac{D}{a_{22}c_{21}-a_{21}c_{22}}yz,c_{21}z^2+c_{22}xz+c_{12}yz\right).$$

Finally,  with  changes of coordinates in the target we get $j^2f(0)=(x,y,xz,yz,z^2)$.

On the other hand, if $D=0$, then $Z$, $W$ and $T$ are linearly dependent, i.e., there exist $A,\ ,B,\,C\neq0$ such that $AZ+BW+CT=0$. Therefore, taking
$$\widetilde{X}=X,\,\,\,\widetilde{Y}=Y,\,\,\,\widetilde{Z}=Z,\,\,\,\widetilde{W}=W\,\,\,\mbox{and}\,\,\,\widetilde{T}=AZ+BW+CT,$$ we obtain $$j^2f(0)\sim(x,y,a_{21}z^2+a_{22}xz+a_{12}yz,b_{21}z^2+b_{22}xz+b_{12}yz,0).$$

Consider $\alpha_2=\left(
                     \begin{array}{ccc}
                       a_{21} & a_{22} & a_{12} \\
                      b_{21} & b_{22} & b_{12}
                     \end{array}
                   \right)
$ and suppose that rank $(\alpha_2)=2$. We have the following possibilities:
\item(i) If $(a_{21},b_{21})\neq(0,0)$, $j^2f(0)\sim(x,y,z^2,xz,0)$. Indeed, without loss of generality  we can suppose $a_{21}\neq0$. Making $$\overline{X}=X,\,\,\,\overline{Y}=Y,\,\,\,\overline{Z}=Z,\,\,\,
\overline{W}=W-\frac{b_{21}}{a_{21}}Z\,\,\,\mbox{and}\,\,\,\overline{T}=T,$$ we get $$j^2f(0)\sim(x,y,a_{21}z^2+a_{22}xz+a_{12}yz,\widetilde{b}_{22}xz+\widetilde{b}_{12}yz,0)$$ and with a change in the source $$x=X',\,\,y=Y'\,\,\mbox{and}\,\,z=Z'-\frac{a_{22}}{2a_{21}}x-\frac{a_{12}}{2a_{21}}y,$$ we get $j^2f(0)\sim(x,y,z^2,\widetilde{b}_{22}xz+\widetilde{b}_{12}yz,0)$. Since rank$(\alpha_2)=2$, we can suppose $\widetilde{b}_{22}\neq0$ and making a change in the source, we get $j^2f(0)\sim(x-\frac{\widetilde{b}_{12}}{\widetilde{b}_{22}}y,y,z^2,xz,0)$. Finally with one more change of coordinates in the target we have the result.

\item(ii) If $(a_{21},b_{21})=(0,0)$, using the fact that rank$(\alpha_2)=2$ and using appropriate changes we get $j^2f(0)\sim(x,y,xz,yz,0)$.

When $D=0$ and rank$(\alpha)=1$, we can take $j^2f(0)\sim(x,y,a_{21}z^2+a_{22}xz+a_{12}yz,0,0)$. In such case we obtain $j^2f(0)\sim(x,y,z^2,0,0)$ and $j^2f(0)\sim(x,y,xz,0,0)$. Lastly, when $D=0$ and rank$(\alpha)=0$, we have $j^2f(0)\sim(x,y,0,0,0)$.

Therefore, we obtain the desired six $\mathcal{A}^2$-orbits in $\Sigma^1J^2(3,5)$.
  \end{proof}

\begin{defn}
Let $M$ be a 3-manifold with a singularity of corank 1 at $p\in M$. We say that $p$ is non degenerate if $D(p)\neq0$.
\end{defn}

\begin{rem}\label{rem:obs orbitas}
The topological type of the curvature parabola for corank 1 surfaces in $\mathbb{R}^n$, $n=3,4$ is a complete invariant for the $\mathcal{A}^2$-classification of 2-jets in $\Sigma^1J^2(2,n)$, ie, it distinguishes the four orbits in both cases, as shown in Theorem 3.6 and Theorem 2.5 of \cite{Pedro/Cidinha/Raul,luciana}, respectively. However, in the case of a singular 3-manifold in $\mathbb{R}^5$, we do not have a similar result. For example, if $j^2f(0)\sim_{\mathcal{A}^2} (x,y,xz,yz,z^2)$, then this does not imply that the curvature locus $\Delta_p$ is always a paraboloid. Indeed, if $f(x,y,z)=(x,y,xz+y^2,yz,z^2)$ so that  $j^2f(0)\sim_{\mathcal{A}^2} (x,y,xz,yz,z^2)$, but the curvature locus $\Delta_p=\{(0,0,2b^2+2ac,2bc,2c^2)\,|\,a^2+b^2=1\}$ is not a paraboloid. Figure \ref{rem-sing} illustrates the curvature locus $\Delta_p$.


%
%

\begin{figure}[!htb]\hspace{2.5cm}
\begin{minipage}[b]{0.3\linewidth}
\includegraphics[scale=0.2]{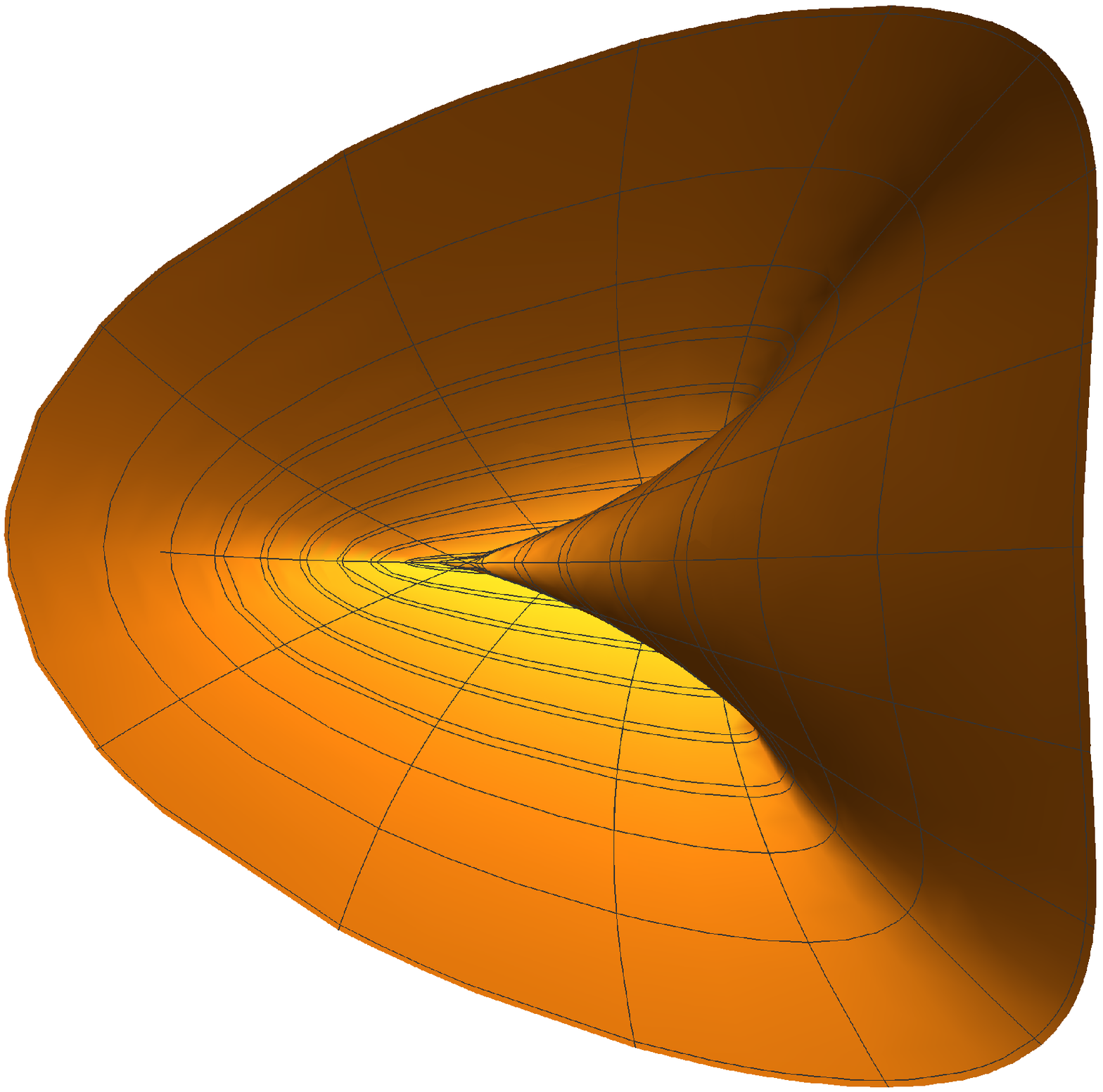}
\end{minipage} \hfill
\begin{minipage}[b]{0.3\linewidth}\hspace{-2cm}
\includegraphics[scale=0.2]{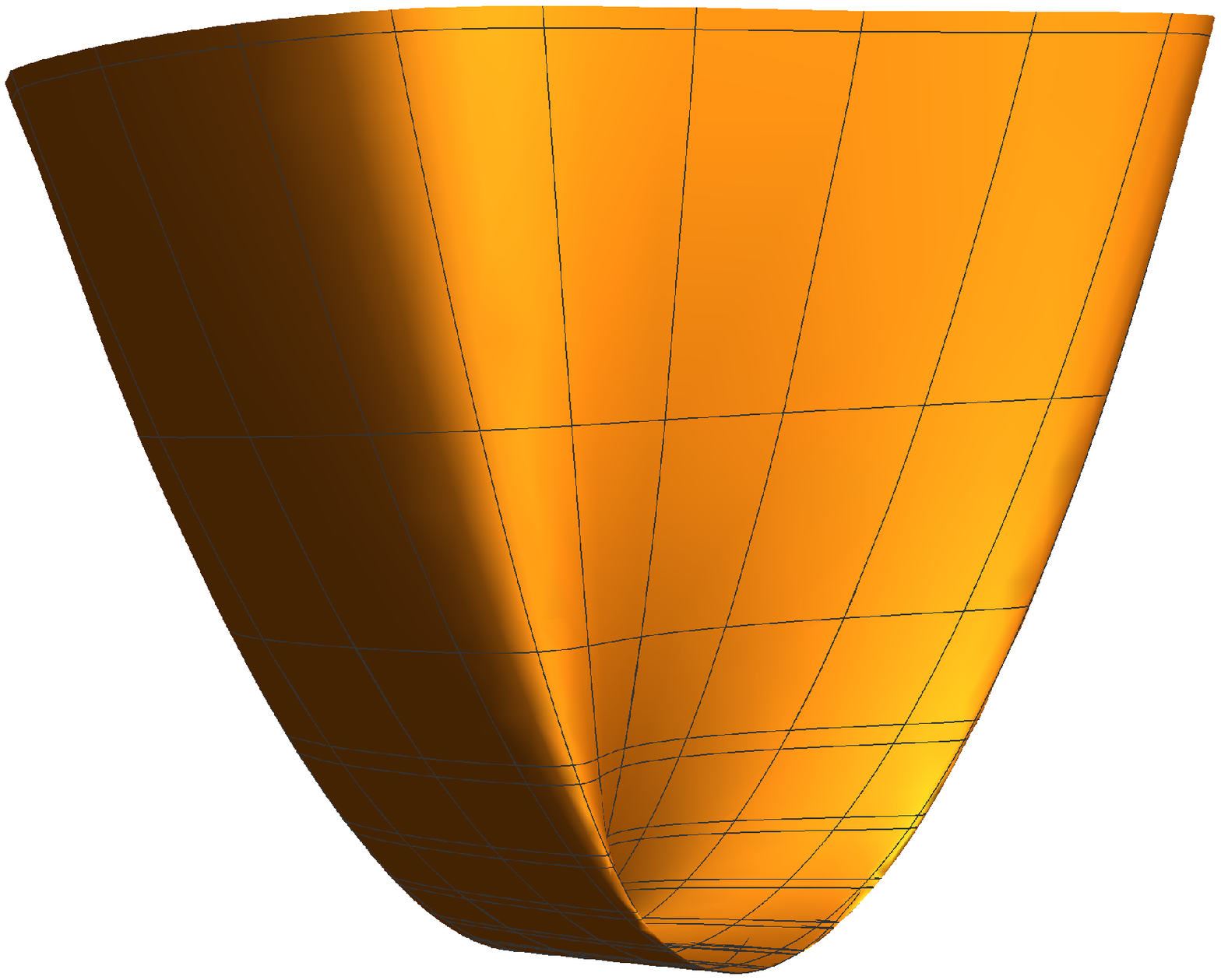}
\end{minipage}
\caption{Top and front views of $\Delta_p$ in Remark \ref{rem:obs orbitas}}
\label{rem-sing}
\end{figure}

\end{rem}

\begin{rem}\label{rem:classification}
Given the $2$-jet

$$\begin{array}{ll}
    j^2f(0)= & (x,y,a_{20}x^2+a_{11}xy+a_{02}y^2+a_{21}z^2+a_{22}xz+a_{12}yz,b_{20}x^2+b_{11}xy+b_{02}y^2+b_{21}z^2 \\
     & +b_{22}xz+b_{12}yz,c_{20}x^2+c_{11}xy+c_{02}y^2+c_{21}z^2+c_{22}xz+c_{12}yz),
  \end{array}
$$

\noindent of corank 1, Table \ref{condicoes} gives conditions on the coefficients
to identify when the $2$-jet is equivalent to one of the six normal forms of Proposition \ref{prop:orbitas},
where  $D=\det(\alpha)$ with $\alpha=\left(
                    \begin{array}{ccc}
                      a_{21} & a_{22} & a_{12} \\
                      b_{21} & b_{22} & b_{12} \\
                      c_{21} & c_{22} & c_{12} \\
                    \end{array}
                  \right)$,
as in (\ref{Det}).

\begin{table}[h]
\caption{Conditions over the coefficients of the $2$-jet for the $\mathcal{A}^{2}$-classification}
\centering
{\begin{tabular}{ccc}
\hline
$\mathcal{A}^{2}$-normal form & Conditions\\
\hline
$(x,y,xz,yz,z^2)$ &  $D\neq 0$ \\ \cr
$(x,y,z^2,xz,0)$ & $ D=0, \,\,\,  rank (\alpha)=2\,\,\, \mbox{and}\,\,\, a_{21}^2+b_{21}^2+c_{21}^2>0$ \\ \cr
$(x,y,xz,yz,0)$ & $D=0, \,\,\,  rank (\alpha)=2\,\,\, \mbox{and}\,\,\, a_{21}=b_{21}=c_{21}=0$\\ \cr
$(x,y,z^2,0,0)$ & $D=0, \,\,\,  rank (\alpha)=1\,\,\, \mbox{and}\,\,\, a_{21}^2+b_{21}^2+c_{21}^2>0$\\ \cr
$(x,y,xz,0,0)$ & $D=0, \,\,\,  rank (\alpha)=1\,\,\, \mbox{and}\,\,\, a_{21}=b_{21}=c_{21}=0$ \\ \cr
$(x,y,0,0,0)$ & $a_{21}= a_{22}=a_{12}=b_{21}= b_{22}=b_{12}=c_{21}= c_{22}=c_{12}=0$. \cr
\hline
\end{tabular}
}
\label{condicoes}
\end{table}


\end{rem}

Although, in general, the topological type of the curvature locus of corank $1$ $3$-manifolds in $\mathbb{R}^5$ is not an invariant of the $\mathcal{A}^2$-classification in $\Sigma^1J^2(3,5)$, we do have a partial result for a special class of corank $1$ $3$-manifold in $\mathbb{R}^5$.

\begin{theo}\label{teo-partial}
Let $M\subset\mathbb{R}^5$ be a corank $1$ $3$-manifold at the origin $p$, locally parametrised by $f:(\mathbb{R}^3,0)\rightarrow(\mathbb{R}^5,0)$, such that
$$f(x,y,z)=(x,y,a_{21}z^2+a_{22}xz+a_{12}yz,b_{21}z^2+b_{22}xz+b_{12}yz,c_{21}z^2+c_{22}xz+c_{12}yz)+o(3).$$
Then,
\begin{itemize}
\item[(i)] $j^2f(0)\sim_{\mathcal{A}^2}(x,y,xz,yz,z^2)\Leftrightarrow\Delta_p$ is a substantial surface;
\item[(ii)] $j^2f(0)\sim_{\mathcal{A}^2}(x,y,z^2,xz,0)\Leftrightarrow\Delta_p$ is a planar region;
\item[(iii)] $j^2f(0)\sim_{\mathcal{A}^2}(x,y,xz,yz,0)\Leftrightarrow\Delta_p$ is a plane;
\item[(iv)] $j^2f(0)\sim_{\mathcal{A}^2}(x,y,z^2,0,0)\Leftrightarrow\Delta_p$ is a half-line;
\item[(v)] $j^2f(0)\sim_{\mathcal{A}^2}(x,y,xz,0,0)\Leftrightarrow\Delta_p$ is a line;
\item[(vi)] $j^2f(0)\sim_{\mathcal{A}^2}(x,y,0,0,0)\Leftrightarrow\Delta_p$ is a point.
\end{itemize}
\end{theo}
\begin{proof}
Table \ref{condicoes} provides conditions over the coefficients on the parametrisation $f$ in order to determine the $\mathcal{A}^2$-orbit
of its $2$-jet.
Taking $f$ as above and the orthonormal frame $\{e_3,e_4,e_5\}$ of $N_pM$, the matrix of the second fundamental form at the origin $p$ is
$$
\left(
  \begin{array}{cccccc}
    0 & 0 & 0 & 2a_{21} & a_{22} & a_{12} \\
    0 & 0 & 0 & 2b_{21} & b_{22} & b_{12} \\
    0 & 0 & 0 & 2c_{21} & c_{22} & c_{12} \\
  \end{array}
\right)
$$
and also $\mbox{rank}(\alpha)=\mbox{rank}(\mathcal{II})$.

The curvature locus $\Delta_p$ is a substantial surface iff $D\neq0$. Indeed, $D=0$ if and only if $\mbox{rank}(\mathcal{II})\leq2$, which is equivalent to the image of the second fundamental form being contained in a plane. Hence, $j^2f(0)\sim_{\mathcal{A}^2}(x,y,xz,yz,z^2)$ iff $\Delta_p$ is a substantial surface.

When $D=0$ and the rank of the second fundamental form is $2$, $f$ can be written as
$$f(x,y,z)=(x,y,f_1,f_2,\lambda f_1+\mu f_2),$$
where $\lambda,\mu\in\mathbb{R}$ and $f_i\in\mathcal{M}_3^2$ with no quadratic terms in $x^2$, $xy$ or $y^2$. In this case, $\Delta_p$ can be parametrised by
$$\eta(x,y,z)=(\underbrace{2a_{21}z^2+2a_{22}xz+2a_{12}yz}_{a(x,y,z)},\underbrace{2b_{21}z^2+2b_{22}xz+2c_{12}yz}_{b(x,y,z)},\lambda a(x,y,z)+\mu b(x,y,z))$$
where $x^2+y^2=1$ and $z\in\mathbb{R}$.
Hence the curvature locus is contained in the plane generated by the vectors $(1,0,\lambda)$ and $(0,1,\mu)$ in $N_pM$.

In the previous conditions, consider $a_{21}^2+a_{22}^2+a_{12}^2>0$. Suppose $a_{21}\neq0$ and write
$$a(x,y,z)=a_{21}z^2+z(a_{22}x+a_{12}y).$$
Since $(a_{22}x+a_{12}y)$ is bounded and $a_{21}\neq0$, the image of the function $a$ does not cover all the real. The same analysis can be applied to
the other functions and $\Delta_p$ does not cover the whole plane. Therefore $\Delta_p$ is a planar region. However, if $a_{21}=a_{22}=a_{12}=0$, the curvature locus is parametrised by
$$\eta(x,y,z)=z^2(\underbrace{2a_{22}x+2a_{12}y}_{a(x,y)},\underbrace{2b_{22}x+2c_{12}y}_{b(x,y)},\lambda a(x,y)+\mu b(x,y))$$
where $z\in\mathbb{R}$ and the second part is a closed curve. Hence, $\Delta_p$ is the whole plane.

When $\mbox{rank}(\mathcal{II})=1$, the image os the second fundamental form is contained in a line and the proof of the cases (iv) and (v) are analogous to the previous. Finally, $\mbox{rank}(\mathcal{II})=0$ iff $\Delta_p$ is only a point.
\end{proof}

Given $M\subset\mathbb{R}^5$ a corank $1$ $3$-manifold at the origin $p$, locally parametrised by $f:(\mathbb{R}^3,0)\rightarrow(\mathbb{R}^5,0)$
it is possible to write the $2$-jet of $f$ at the origin, $j^2f(0)$, in the special form of Theorem \ref{teo-partial} just using the action of the group
$\mathcal{A}^2$.

\begin{ex}
The previous result is not true in general. Consider, for instance, the corank 1 $3$-manifold
given by $f(x,y,z)=(x,y,x^2+z^2,xy+xz,y^2)$. Notice that $j^2f(0)$ is $\mathcal{A}^2$-equivalent
to $(x,y,z^2,xz,0)$, nevertheless, $\Delta_p$ is a substantial surface.
\end{ex}

We denote by $\mathcal{R}^2$ the group of $2$-jets of diffeomorphisms from $(\mathbb{R}^3,0)$ to $(\mathbb{R}^3,0)$ and by $\mathcal{O}(5)$ the group of linear isometries of $\mathbb{R}^5$. Then, $\mathcal{R}^2\times \mathcal{O}(5)$ is a subgroup of $\mathcal{A}^2$ which also acts on $\Sigma^1J^2(3,5)$, the subspace of $2$-jets of corank 1.

We list below the parametrisations obtained with changes of coordinates in source and target preserving the geometry of the image.
To obtain such parametrisations we use the classification given in Remark \ref{rem:classification}.

\begin{prop}\label{prop:formanormal}
Let $f:(\mathbb{R}^3,0)\rightarrow (\mathbb{R}^5,0)$ be a corank 1 map germ. Using smooth changes of coordinates in the source and isometries in the target, we can reduce $f$ to the form
\begin{itemize}
\item [\rm{(a)}] $(x,y,z)\mapsto (x,y,a_1x^2+a_3y^2+xz+a_6yz+p(x,y,z),b_1x^2+b_2xy+b_3y^2+b_6yz+q(x,y,z),c_1x^2+c_2xy+c_3y^2+c_4z^2+c_5xz+c_6yz+r(x,y,z))$ if $j^2f(0)\sim_{\mathcal{A}^2} (x,y,xz,yz,z^2)$;

\item [\rm{(b)}] $(x,y,z)\mapsto (x,y,a_1x^2+a_2xy+a_3y^2+a_4z^2+a_5xz+p(x,y,z),b_1x^2+b_2xy+b_3y^2+xz+q(x,y,z),c_1x^2+c_2xy+c_3y^2+r(x,y,z))$ if  $j^2f(0)\sim_{\mathcal{A}^2}  (x,y,z^2,xz,0)$;

\item [\rm{(c)}] $(x,y,z)\mapsto (x,y,a_1x^2+a_3y^2+xz + p(x,y,z),b_1x^2+b_3y^2+b_5xz+b_6yz+q(x,y,z),c_1x^2+c_2xy+c_3y^2+r(x,y,z))$ if  $j^2f(0)\sim_{\mathcal{A}^2}  (x,y,xz,yz,0)$;

\item [\rm{(d)}] $(x,y,z)\mapsto (x,y,a_1x^2+a_2xy+a_3y^2+z^2+p(x,y,z),b_1x^2+b_2xy+b_3y^2+q(x,y,z),c_1x^2+c_2xy+c_3y^2+r(x,y,z))$ if  $j^2f(0)\sim_{\mathcal{A}^2}  (x,y,z^2,0,0)$;

\item [\rm{(e)}] $(x,y,z)\mapsto (x,y,a_3y^2+xz+p(x,y,z),b_1x^2+b_2xy+b_3y^2+q(x,y,z),c_1x^2+c_2xy+c_3y^2+r(x,y,z))$ if  $j^2f(0)\sim_{\mathcal{A}^2}  (x,y,xz,0,0)$;

\item [\rm{(f)}] $(x,y,z)\mapsto (x,y,a_1x^2+a_2xy+a_3y^2+p(x,y,z),b_1x^2+b_2xy+b_3y^2+q(x,y,z),c_1x^2+c_2xy+c_3y^2+r(x,y,z))$ if  $j^2f(0)\sim_{\mathcal{A}^2}  (x,y,0,0,0)$,\\
\end{itemize}
\noindent where $a_i,b_i,c_i$, $i=1,\ldots,6$ are constants, $a_4,c_4>0$, $b_6\neq0$ and $p,q,r\in\mathcal{M}_3^3$.
\end{prop}
\begin{proof}
Let $f$ be a corank 1 map germ such that $j^2f(0)\sim_{\mathcal{A}^2}(x,y,xz,yz,z^2)$, then we have

$$
\begin{array}{ll}
j^2f(0)= & (x,y,a_{20}x^2+a_{11}xy+a_{02}y^2+a_{21}z^2+a_{22}xz+a_{12}yz,b_{20}x^2+b_{11}xy+b_{02}y^2\\
& +b_{21}z^2+b_{22}xz+b_{12}yz,c_{20}x^2+c_{11}xy+c_{02}y^2+c_{21}z^2+c_{22}xz+c_{12}yz).
\end{array}
$$

So it is possible to take $c_{21}\neq0$. A change of coordinates in the variable $z$, making $z=-\dfrac{c_{22}}{2c_{21}}x-\dfrac{c_{12}}{2c_{21}}y+z'$, transforms $j^2f(0)$ into
$$
\begin{array}{l}
(x,y,\widetilde{a}_{20}x^2+\widetilde{a}_{11}xy+\widetilde{a}_{02}y^2+\widetilde{a}_{21}z^2+\widetilde{a}_{22}xz+\widetilde{a}_{12}yz,\widetilde{b}_{20}x^2+\widetilde{b}_{11}xy+
\widetilde{b}_{02}y^2\\
+\widetilde{b}_{21}z^2+\widetilde{b}_{22}xz+\widetilde{b}_{12}yz,
\widetilde{c}_{20}x^2+\widetilde{c}_{11}xy+\widetilde{c}_{02}y^2+c_{21}z^2).
\end{array}
$$

Since $D\neq0$, we have $(\widetilde{a}_{22})^2+(\widetilde{b}_{22})^2>0$ and $\widetilde{a}_{22}\widetilde{b}_{12}-\widetilde{b}_{22}\widetilde{a}_{12}\neq0$.
Without loss of generality we can assume $\widetilde{a}_{22}\neq0$ and take the rotation in $\mathbb{R}^5$ given by

$$R_1=\left(
    \begin{array}{ccccc}
      1 & 0 & 0 & 0 & 0 \\
      0 & 1 & 0 & 0 & 0 \\
      0 & 0 & \cos(\theta_1) & \sin(\theta_1)  & 0 \\
      0 & 0 & - \sin(\theta_1) & \cos(\theta_1) &   0 \\
      0 & 0 & 0 & 0&1 \\
    \end{array}
  \right)
$$

\noindent through the angle $\theta_1=\arctan\left(\dfrac{\widetilde{b}_{22}}{\widetilde{a}_{22}}\right)$, that transforms $j^2f(0)$ into

$$
\begin{array}{l}
(x,y,{a}'_{20}x^2+{a}'_{11}xy+{a}'_{02}y^2+{a}'_{21}z^2+{a}'_{22}xz+{a}'_{12}yz,{b}'_{20}x^2+{b}'_{11}xy+
{b}'_{02}y^2+{b}'_{21}z^2+{b}'_{12}yz,\\
\widetilde{c}_{20}x^2+\widetilde{c}_{11}xy+\widetilde{c}_{02}y^2+c_{21}z^2),
\end{array}
$$
where ${a}'_{22}>0$ and ${b}'_{12}\neq0$.

Now let
$$R_2=\left(
    \begin{array}{ccccc}
      1 & 0 & 0 & 0 & 0 \\
      0 & 1 & 0 & 0 & 0 \\
      0 & 0 & 1 & 0 & 0 \\
      0 & 0 & 0 & \cos(\theta_2) &-\sin(\theta_2)     \\
      0 & 0 & 0 & \sin(\theta_2) &\cos(\theta_2) \\
    \end{array}
  \right)
$$
\noindent be the rotation through the angle $\theta_2=\arctan\left(\dfrac{{b}'_{21}}{{c}_{21}}\right)$. Thus,
$$
\begin{array}{ll}
j^2f(0)\sim & (x,y,{a}'_{20}x^2+{a}'_{11}xy+{a}'_{02}y^2+{a}'_{21}z^2+{a}'_{22}xz+{a}'_{12}yz,\\
 & {b}''_{20}x^2+{b}''_{11}xy+{b}''_{02}y^2+{b}''_{12}yz,{c}'_{20}x^2+{c}'_{11}xy+{c}'_{02}y^2+c'_{21}z^2),
\end{array}
$$

whit ${b}''_{12}\neq0$.
Furthermore, considering $\theta_3=\arctan\left(\dfrac{{a}'_{21}}{{c}'_{21}}\right)$ and the rotation in $\mathbb{R}^5$

$$R_3=\left(
    \begin{array}{ccccc}
      1 & 0 & 0 & 0 & 0 \\
      0 & 1 & 0 & 0 & 0 \\
      0 & 0 & \cos(\theta_3) & 0 & -\sin(\theta_3) \\
      0 & 0 & 0 & 1 &   0  \\
      0 & 0 & \sin(\theta_3) & 0 &\cos(\theta_3) \\
    \end{array}
  \right),$$
 \noindent we transform $j^2f(0)$ into
$$
\begin{array}{l}
(x,y,\widetilde{\widetilde{a}}_{20}x^2+\widetilde{\widetilde{a}}_{11}xy+\widetilde{\widetilde{a}}_{02}y^2+\widetilde{\widetilde{a}}_{22}xz+\widetilde{\widetilde{a}}_{12}yz,
 {b}''_{20}x^2+{b}''_{11}xy+{b}''_{02}y^2+{b}''_{12}yz,\\
{c}''_{20}x^2+{c}''_{11}xy+{c}''_{02}y^2+c''_{21}z^2+c''_{22}xz+c''_{12}yz),
\end{array}
$$
where $\widetilde{\widetilde{a}}_{22},c''_{21}>0$.
Finally with the change of coordinates in the source $z=-\dfrac{\widetilde{\widetilde{a}}_{11}}{\widetilde{\widetilde{a}}_{22}}y+z'$, we have
$$
\begin{array}{ll}
j^2f(0) \sim & (x,y,a_1x^2+a_3y^2+xz+a_6yz,b_1x^2+b_2xy+b_3y^2+b_6yz,\\
 & c_1x^2+c_2xy+c_3y^2+c_4z^2+c_5xz+c_6yz),
\end{array}
$$
where $c_4>0$ and $b_6\neq0$.
\end{proof}

The next result shows that two $2$-jets of corank 1 map germs $(\mathbb{R}^3,0)\rightarrow (\mathbb{R}^5,0)$ are equivalent under the action $\mathcal{R}^2\times \mathcal{O}(5)$ if and only if there exists an isometry between the normal spaces preserving the respective curvature locus.

\begin{theo}\label{teo:isometria}
Let $M_1,M_2\subset\mathbb{R}^5$ be two $3$-manifolds with a corank 1 singularity at points $p_1\in M_1$ and $p_2\in M_2$, parametrised by $f$ and $g$, respectively. The $2$-jets $j^2f(0)$, $j^2g(0)\in \Sigma^1J^2(3,5)$ are equivalent under the action $\mathcal{R}^2\times \mathcal{O}(5)$ if and only if there is a linear isometry $\phi:N_{p_1}M_1\rightarrow N_{p_2}M_2$ such that $\phi (\Delta_{p_1}(M_1))=\Delta_{p_2}(M_2)$.
\end{theo}

The proof of this result is analogous to the proofs of Theorem 1.7 in \cite{luciana} and Theorem 3.8 in \cite{Pedro/Cidinha/Raul}. However, the proof is very long and can not be carried out by hand. Therefore, we present it in Appendix \ref{appendix}.

\section{Geometry of the curvature locus}\label{sec:geometria locus}

The aim of this section is to study the geometry of the curvature locus of a singular 3-manifold in $\mathbb{R}^5$. For this we  relate this curvature locus with the curvature locus  of a smooth 3-manifold  in $\mathbb{R}^6$, given in \cite{carmen}, through a local lifting. In what follows we review definitions and results from \cite{carmen}.

\subsection{Curvature locus of a smooth 3-manifold in $\mathbb{R}^6$}\label{subsec:locus carmem}
Let $\overline{M}$ be a smooth 3-manifold  in $\mathbb{R}^6$  parametrised by $\overline{f}:U\subset \mathbb{R}^3\rightarrow\mathbb{R}^6$ and a point $\overline{p}\in \overline{M}$.
For each $\overline{p}=\overline{f}(x,y,z)$, the \emph{tangent space} $T_{\overline{p}}\overline{M}$ is the $3$-dimensional vector space generated by $B^{\overline{f}}=\{\frac{\partial\overline{f}}{\partial x}(\overline{p}),\frac{\partial\overline{f}}{\partial y}(\overline{p}),\frac{\partial\overline{f}}{\partial z}(\overline{p})\}$. Also, the orthonormal frame that generates the \emph{normal space} $N_{\overline{p}}\overline{M}\equiv\mathbb{R}^{3}$ is $\{\textbf{e}_{1},\textbf{e}_{2},\textbf{e}_{3}\}$ such that the orientation of the frame
$$\left\{\frac{\partial\overline{f}}{\partial x}(\overline{p}),\frac{\partial\overline{f}}{\partial y}(\overline{p}),\frac{\partial\overline{f}}{\partial z}(\overline{p}),\textbf{e}_{1},\textbf{e}_{2},\textbf{e}_{3}\right\}$$
is the same as the one of $\mathbb{R}^{6}$.

Given a normal field $\nu$ in $\overline{M}$, consider the linear map
$$d\nu:T_{\overline{p}}\overline{M}\rightarrow T_{\overline{p}}\mathbb{R}^{6}\equiv T_{\overline{p}}\overline{M}\oplus N_{\overline{p}}\overline{M}$$
and the canonical projection $\pi^{\perp}:T_{\overline{p}}\overline{M}\oplus N_{\overline{p}}\overline{M}\rightarrow N_{\overline{p}}\overline{M}$. The map
$$S_{\overline{p}}^{\nu}=-\pi^{\perp}\circ d\nu:T_{\overline{p}}\overline{M}\rightarrow T_{\overline{p}}\overline{M}$$
is called the $\nu$-\emph{shape operator} at $\bar{p}$.
The \emph{second fundamental form of $\overline{M}$ along the normal field $\nu$} is the bilinear map
$$II_{\overline{p}}^{\nu}:T_{\overline{p}}\overline{M}\times T_{\overline{p}}\overline{M}\rightarrow\mathbb{R},\ II_{\overline{p}}^{\nu}(\textbf{v},\textbf{w})=\langle\nu,d^2\bar{f}(\textbf{v},\textbf{w})\rangle,$$
for all $\textbf{v},\textbf{w}\in T_{\overline{p}}\overline{M}$, where $d^2\bar{f}$ denotes the second derivative of $\overline{f}$.

Let $\overline{p}\in\overline{M}\subset\mathbb{R}^{6}$ and $v\in\mathbb{S}^{2}\subset T_{\overline{p}}\overline{M}$. We denote by $\gamma_v$ the normal section of $\overline{M}$ in the direction $v$, that is, $\gamma_v=\overline{M}\cap\mathbb{H}_v$, with $\mathbb{H}_v=\{\lambda v\}\oplus N_{\overline{p}}\overline{M}$ being a $4$-space through $\overline{p}$ in $\mathbb{R}^{6}$. The normal curvature vector $\eta(v)$ of $\gamma_v$ at $\overline{p}$, given by the projection of $\gamma''_v$ in $N_{\overline{p}}\overline{M}$. Varying $v$ in all unit sphere $\mathbb{S}^2\subset T_{\overline{p}}\overline{M}$, we obtain a surface contained in the normal space $N_{\overline{p}}\overline{M}$. Such surface is called the \emph{curvature locus of $\overline{M}$ at $\overline{p}$}.

The curvature locus can also be seen as the image of the unit sphere $\mathbb{S}^{2}\subset T_{\overline{p}}\overline{M}$ via
the second fundamental form of $\overline{M}$ at $\overline{p}$.

In \cite{carmen}, the authors prove that taking spherical coordinates in the unit tangent sphere $\mathbb{S}^{2}\subset T_{\overline{p}}\overline{M}$, the curvature locus is the image of the map

$$\begin{array}{cccc}
  \overline{\eta}: & \mathbb{S}^2\subset T_{p}\overline{M} & \rightarrow & N_{p}\overline{M}  \\
   & (\theta,\phi) & \mapsto& \overline{\eta}(\theta,\phi)
\end{array},$$

where

$$
\begin{array}{ll}
 \overline{\eta}(\theta,\phi)= & H+(1+3\cos(2\phi))B_1+\cos(2\theta)(\sin(\phi))^2B_2+\sin(2\theta)(\sin(\phi))^2B_3\\
 & +\cos(\theta)\sin(2\phi)B_4+\sin(\theta)\sin(2\phi)B_5,
\end{array}
$$

\noindent with $H=\frac{1}{3}(\overline{f}_{xx}+\overline{f}_{yy}+\overline{f}_{zz})$, $B_1=\frac{1}{12}(-\overline{f}_{xx}-\overline{f}_{yy}+2\overline{f}_{zz})$, $B_2=\frac{1}{2}(\overline{f}_{xx}-\overline{f}_{yy})$, $B_3=\overline{f}_{xy}$, $B_4=\overline{f}_{xz}$ and $B_5=\overline{f}_{yz}$, where $\overline{f}_{xx}=\frac{\partial^2\overline{f}}{\partial x^2}$, etc.

The value of the normal field $H$ at  $\overline{p}\in\overline{M}$, is the mean curvature vector of $\overline{M}$ at $\overline{p}$.

The first normal of $\overline{M}$ at $\overline{p}$ is given by $$N_{\overline{p}}^1\overline{M}=\langle \overline{f}_{xx},\overline{f}_{xy},\overline{f}_{xz},\overline{f}_{yy},\overline{f}_{yz},\overline{f}_{zz}\rangle_{\overline{p}}.$$
This space coincides with the subspace $\langle H,B_1,B_2,B_3,B_4,B_5\rangle$. Denote by $Af\overline{f}_{\overline{p}}$ the affine hull of the curvature locus in $N_{\overline{p}}^1\overline{M}$ and by $\overline{E}_{\overline{p}}=\langle B_1,B_2,B_3,B_4,B_5\rangle$ the linear subspace of $N_{\overline{p}}^1\overline{M}$ which is parallel to $Af\overline{f}_{\overline{p}}$.

An alternative expression for the locus is  given by $$\overline{\eta}(u,v,w)=H+(4-6u^2-6v^2)B_1+(u^2-v^2)B_2+2uvB_3+2uwB_4+2vwB_5,$$ where $u^2+v^2+w^2=1$.

In \cite{carmen} the authors interpret the curvature locus as the image of the classical Veronese surface of order 2 through a convenient linear map. We recall the expression of the  Veronese surface of order 2, which is given by the image of the 2-sphere in $\mathbb{R}^3$ through the map $\xi:\mathbb{R}^3\rightarrow \mathbb{R}^6$, defined by $\xi(u,v,w)=(u^2,v^2,w^2,\sqrt{2}uv,\sqrt{2}uw,\sqrt{2}vw)$.
\begin{defn}
Let $\overline{M}$ be a 3-manifold immersed in $\mathbb{R}^6$. We say that a point $\overline{p}\in \overline{M}$ \emph{is of type $\overline{M}_i$}, $i=0,1,2,3$ if $\dim N_{\overline{p}}^1\overline{M}=i$.
\end{defn}
By the previous definition, a point $\overline{p}$ is of type $\overline{M}_i$, $i=0,1,2,3$ if and only if the second fundamental form has rank $i$ at $\overline{p}$.

\begin{theo}{\rm\cite{carmen}}
Given a compact 3-manifold $\overline{M}$ in $\mathbb{R}^6$, there is a residual subset of immersions of $\overline{M}$ into $\mathbb{R}^6$ such that every point of $\overline{M}$ is of type $\overline{M}_3$.
\end{theo}

As we have seen in the previous theorem, a generic immersion of a 3-manifold $\overline{M}$ into $\mathbb{R}^6$ is exclusively made of $\overline{M}_3$ points.
The next theorem describes all the possible topological types for the curvature locus at the points of these manifolds.
To prove it, the authors use information contained in \cite{adam}, where a description
of quadratically parametrised surfaces of order greater that 2 in $\mathbb{R}^3$ is given.

\begin{theo}{\rm\cite{carmen}}\label{teo:principalcarmen}
The curvature locus at a point of type $\overline{M}_3$ of a 3-manifold $\overline{M}$ immersed in $\mathbb{R}^6$ may have one of the following shapes:
\begin{itemize}
\item[\rm{(1)}] A Roman Steiner surface, a Crosscap surface, a Steiner surface of type 5, a Cross-Cup surface, an ellipsoid, or a (compact) cone, provided the mean curvature vector $H(\overline{p})$ lies in $\overline{E}_{\overline{p}}$.
    \item[\rm{(2)}] In the case that the mean curvature vector $H(\overline{p})$ does not lie in $\overline{E}_{\overline{p}}$, the curvature locus is a convex planar region linearly equivalent to one of a following: an elliptic region, a triangle, a (compact) planar cone or a planar projection (of type 1, 2 or 3) of the Veronese surface.
\end{itemize}

\begin{center}
\begin{figure}[!htb]
\begin{minipage}[b]{0.5\linewidth} \hspace{1.9cm}
\includegraphics[scale=0.4]{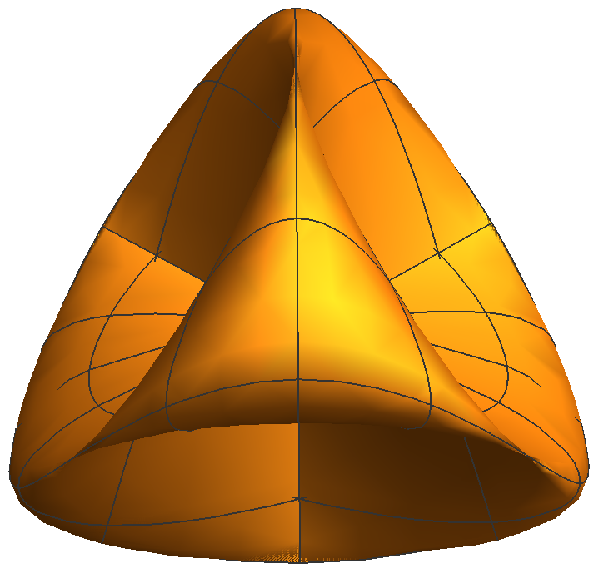}
\caption{Roman steiner.}
\label{fig:steiner1}
\end{minipage} \hfill
\begin{minipage}[b]{0.5\linewidth} \hspace{1.7cm}
\includegraphics[scale=0.4]{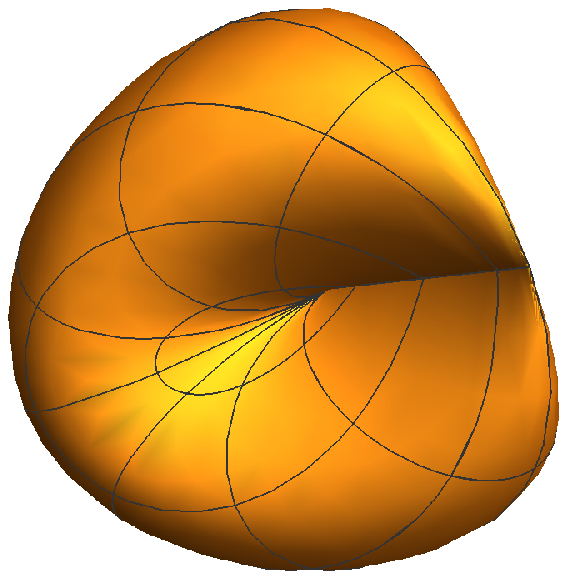}
\caption{Cross-cap.}
\label{fig:cross-cap}
\end{minipage} \hfill \\
\begin{minipage}[b]{0.5\linewidth} \hspace{1.7cm}
\includegraphics[scale=0.4]{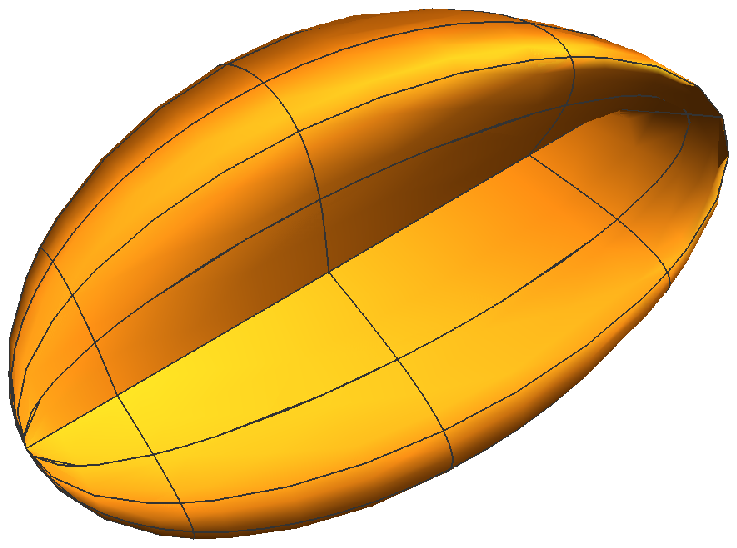}
\caption{Steiner Type $5$.}
\label{fig:tipo5}
\end{minipage} \hfill
\begin{minipage}[b]{0.5\linewidth} \hspace{1.9cm}
\includegraphics[scale=0.4]{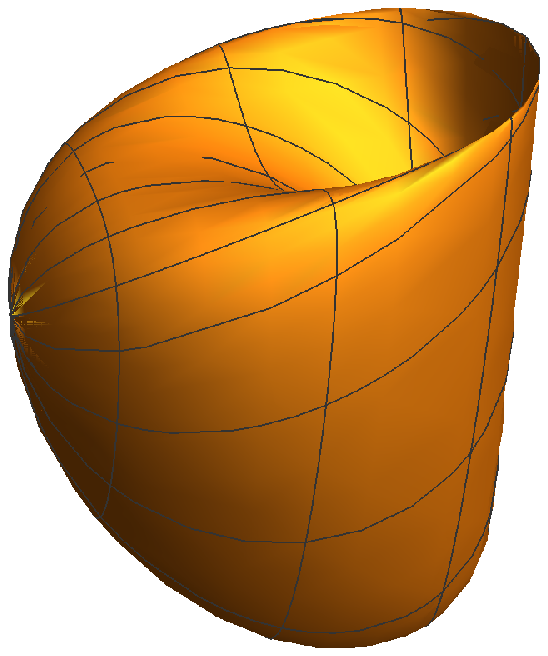}
\caption{Cross-cup.}
\label{fig:cross-cup}
\end{minipage} \hfill
\end{figure}
\end{center}

\begin{figure}[!htb]
\begin{minipage}[b]{0.3\linewidth} \hspace{0.7cm}
\includegraphics[scale=0.35]{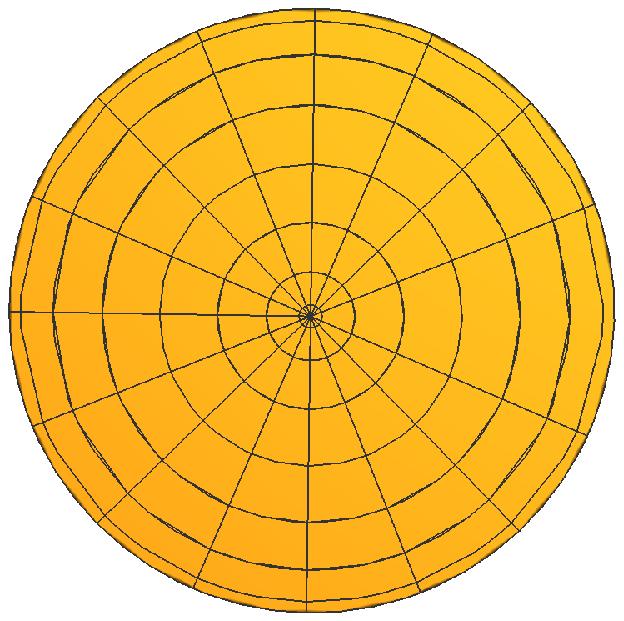}
\caption{Elliptic region}
\label{fig:eliptica}
\end{minipage}
\begin{minipage}[b]{0.3\linewidth}\hspace{0.7cm}
\includegraphics[scale=0.35]{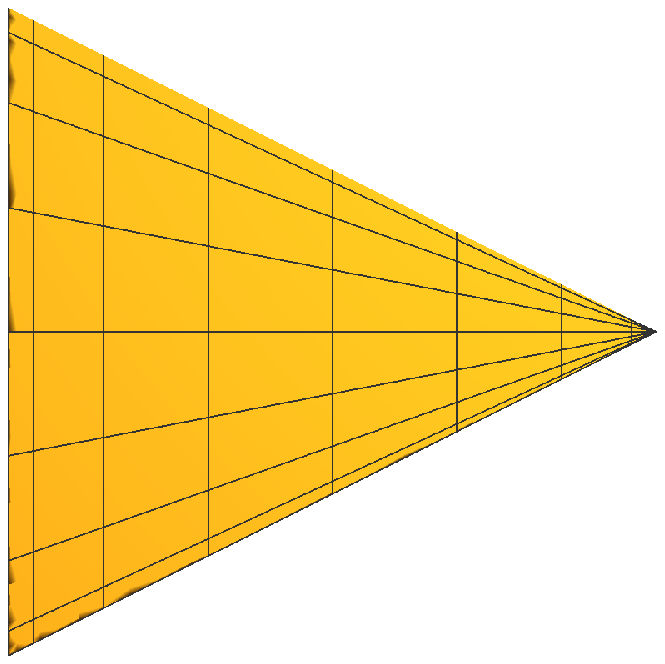}
\caption{Triangle}
\label{fig:triangular}
\end{minipage}
\begin{minipage}[b]{0.3\linewidth} \hspace{0.7cm}
\includegraphics[scale=0.39]{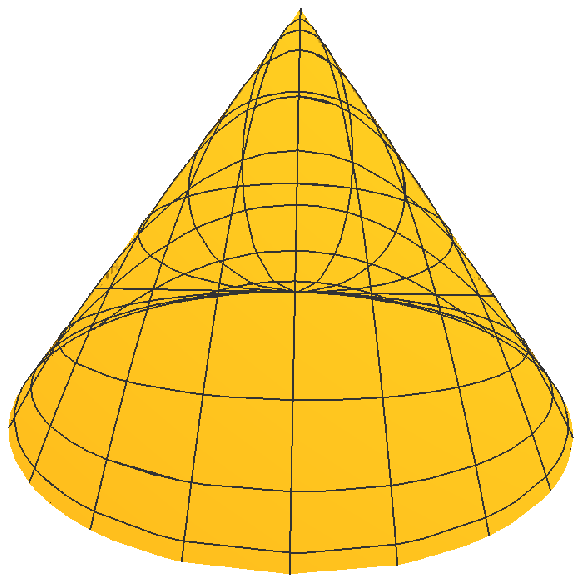}
\caption{Planar cone}
\label{fig:cone}
\end{minipage} \\
\begin{minipage}[b]{0.3\linewidth} \hspace{0.9cm}
\includegraphics[scale=0.35]{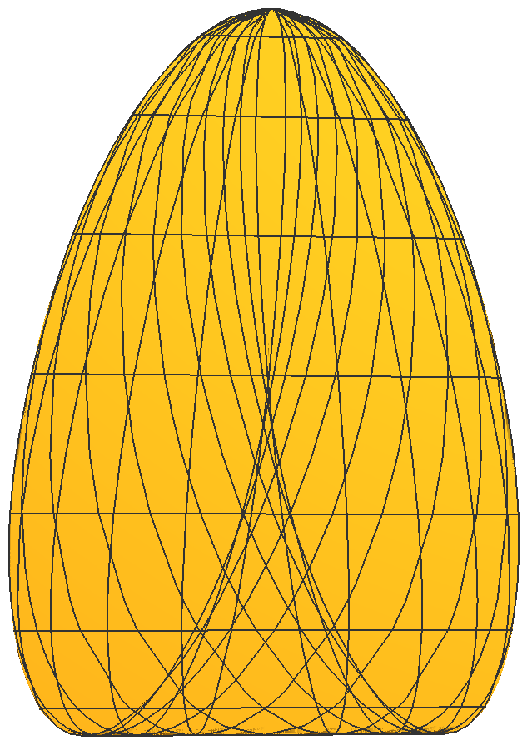}
\caption{Type $1$}
\label{fig:tipo1}
\end{minipage}
\begin{minipage}[b]{0.3\linewidth} \hspace{0.7cm} \vspace{-0.4cm}
\includegraphics[scale=0.35]{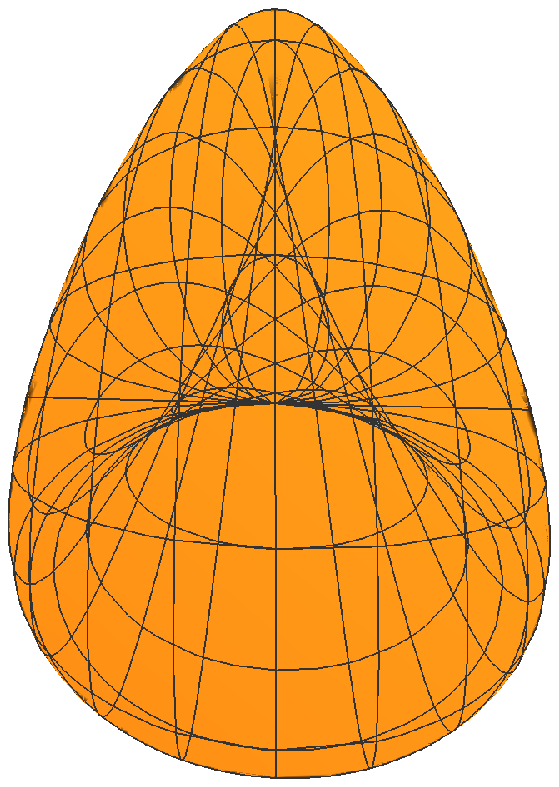}
\caption{Type $2$}
\label{fig:tipo2}
\end{minipage}
\begin{minipage}[b]{0.3\linewidth} \hspace{0.7cm}
\includegraphics[scale=0.35]{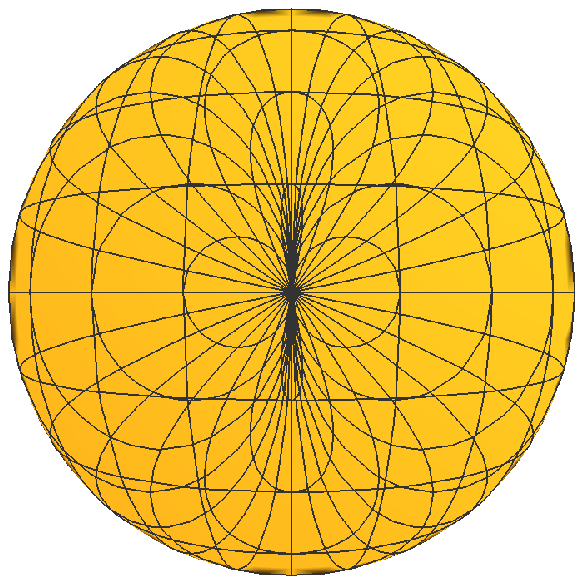}
\caption{Type $3$}
\label{fig:tipo3}
\end{minipage}
\end{figure}

%
\end{theo}

The curvature loci in Theorem \ref{teo:principalcarmen} are projections of the classical Veronese surface of order 2 into a linear 3-space. These surfaces have an implicit form in terms of polynomials of degree at most 4.


Also by \cite{carmen} we have that the curvature locus may be a planar region just at isolated points of generic immersions.

\subsection{Real nets of quadrics}\label{subsec:nets}

The affine geometry of the second fundamental form is the study of invariants
of the action of the affine group $\mathcal{G}=GL(3)\times GL(3)$ in the space of quadratic polynomial mappings $g:\mathbb{R}^3\rightarrow\mathbb{R}^3$. Each
such mapping generates a system of quadrics denominated a ``net of quadrics". One can find in the literature many texts about nets of quadrics (see for instance
\cite{wall/duplesis,wall,wall1,wall2}).
In \cite{wall/duplesis} is presented a classification of the real nets of quadrics with respect to $\mathcal{K}$-equivalence (see \cite{mather} or \cite{wall3} for the definition of the contact group and its action).

Let $H^2(3)$ be the space of homogeneous polynomials of degree 2 in 3 variables. A net of quadrics is a system in $H^2(3)$ generated by 3 polynomials $q_1,q_2,q_3$, where $q_i\in H^2(3)$ for $i=1,2,3$. Associated to each net there is a map germ $q:(\mathbb{R}^3,0)\rightarrow (\mathbb{R}^3,0)$, $q=(q_1,q_2,q_3)$. Denote by $\langle q \rangle=\langle q_1,q_2,q_3\rangle$ the net given by $q=(q_1,q_2,q_3)$, that is, $\langle q\rangle=\{\lambda q_1+\mu q_2+\nu q_3\,|\,\lambda,\mu,\nu\in\mathbb{R}\}$. Let $\Lambda$ be the set of all nets $\langle q\rangle$. It follows from \cite{wall} that there is a Zariski open set of $\Lambda$, denoted by $\Lambda_0$, such that any net $\langle q\rangle$ in $\Lambda_0$ can be taken in the form:
\begin{equation}\label{eq:net1}
\lambda(2xz+y^2)+\mu( 2yz)+\nu(-x^2-2gy^2+cz^2+2gxz),\,\,\, c(c + 9g^2) \neq 0\,\,\, (\mbox{see}\,\, \cite{wall4})
\end{equation}
or in the Hessian form
\begin{equation}\label{eq:net2}
\lambda(x^2+2cyz)+\mu(y^2+2cxz)+\nu(z^2+2cxy), \,\,\,\,\,c(c^3-1)(8c^3+1)\neq 0.
\end{equation}

A net in this set is called a general real net of quadrics. As the normal forms of the generic nets are given by homogeneous polynomial maps of degree 2, and in this case, the corresponding map germ $q=(q_1,q_2,q_3)$ is 2-determined with respect to $\mathcal{K}$-equivalence, it follows that the $\mathcal{K}$-classification coincides with
the classification by the action of the group $\mathcal{G}=GL(3)\times GL(3)$.

The complete classification of quadratic mappings $q\in H^2(3)$ with respect to $\mathcal{G}=GL(3)\times GL(3)$-equivalence can be found
in \cite{wall/duplesis}. The family (\ref{eq:net1}) is labelled $A$, $B$ and $C$ according to the values of the parameters $c$ and $g$.
Table \ref{caseAB} presents the types $A$, $B$ and its subcases. Type $C$ is given by $c=g=0$, and the discriminant for cases $A,B$ and $C$ is
$\Delta=-\lambda^2\nu+(\lambda-2g\nu)(\lambda^2+2g\lambda\nu+(c+g^2)\nu^2)$.

The orbits in the complex case are labelled as follows:
\newline

\begin{tabular}{lccccccccccccccc}
  Name & $A$ & $B$ & $B^*$ & $C$ & $D$ & $D^*$ & $E$ & $E^*$ & $F$ & $F^*$ & $G$ & $G^*$ & $H$ & $I$ & $I^*$ \\
  Codimension & 0 & 1 & 1 & 2 & 2 & 2 & 3 & 3 & 3 & 3 & 4 & 4 & 5 & 7 & 7 \\
\end{tabular}
\newline

The type $A$ depends on a modulus and in the real case it splits four subcases. Type $B$, $B^*$, $C$, $C^*$, $D$, $D^*$, $E$, $E^*$, $F$ and $F^*$ also have subcases.
\newline

\begin{table}[h]
\caption{Orbits $A$ and $B$}
\centering
{\begin{tabular}{cccccc}
  \hline
  $c<-9g^2$ & $c=-9g^2$ & $-9g^2<g<0$ & $c=0$ & $c>0$ &  \\
   & $B_c$ & $A_c$ & $B_a^*$ &  & $g>0$  \\
  $A_b$ &  &  &  & $A_d$ &  \\
   & $B_a$ & $A_a$ & $B_c^*$ &  & $g<0$ \\
  \hline
\end{tabular}
}
\label{caseAB}
\end{table}

The remaining cases are shown in Table \ref{other types}, along with their respective discriminants.

\begin{table}[h]
\caption{Other orbits}
\centering
{\begin{tabular}{ccc}
\hline
Name & Normal form & discriminant\\
\hline
$D_a$ &  $\langle x^2,y^2,z^2+2xy\rangle$ & $\nu(\lambda\mu-\nu^2)$ \\
$D_b,\ D_c$ & $\langle x^2-y^2,2xy,x^2\pm z^2\rangle$ & $\nu(\lambda^2+\lambda\nu+\mu^2)$ \\
$D_a^*$ & $\langle 2xz,2yz,z^2+2xy\rangle$ & $\nu(2\lambda\mu-\nu^2)$\\
$D_b^*,\ D_c^*$ & $\langle 2xz,2yz,x^2+y^2\mp z^2 \rangle$ & $\nu(\lambda^2+\mu^2\pm\nu^2)$\\
$E_a,\ E_b$ & $\langle x^2\pm y^2,2xy,z^2 \rangle$ & $\nu(\lambda^2\mp \mu^2)$ \\
$E_a^*,\ E_b^*$ & $\langle x^2\mp y^2,2xz,2yz \rangle$ & $\lambda(\mu^2\pm\nu^2)$  \\
$F_a,\ F_b$ & $\langle x^2\pm y^2,2xy,2yz \rangle$ & $\lambda\nu^2$  \\
$F_a^*,\ F_b^*$ & $\langle x^2\mp y^2,2xz,z^2 \rangle$ &  $\lambda(\lambda\nu-\mu^2)$ \\
$G$ &  $\langle x^2,y^2,2yz \rangle$ & $\lambda\nu^2$ \\
$G^*$ & $\langle 2xy,2xz,z^2 \rangle$ & $\lambda^2\nu$ \\
$H$ & $\langle x^2,2xy,y^2+2xz \rangle$ & $\nu^3$ \\
$I$ & $\langle x^2,2xy,y^2 \rangle$ & $0$ \\
$I^*$ & $\langle 2xz,2yz,z^2 \rangle$  & $0$ \\
\hline
\end{tabular}
}
\label{other types}
\end{table}

Given a net $q=(q_1,q_2,q_3)$, we can naturally associate the $2$-jet of parametrisation of a smooth $3$-manifold in the Monge form:
$$(x,y,z)\mapsto(x,y,z,q_1(x,y,z),q_2(x,y,z),q_3(x,y,z)),$$
whose second fundamental form at the origin is given by the quadratic map $q$.
A natural question that may arise is the following: given two $\mathcal{G}=GL(3)\times GL(3)$-equivalent nets, $\langle q\rangle$ and $\langle \bar{q}\rangle$, do the respective curvature loci have the same topological type? The answer to that questions is no. The following example illustrates this assertion.

\begin{ex}\label{ex.net}
The $3$-manifold given by $$f(x,y,z)=(x,y,z,x^2+yz,z^2-yz,xz+yz)$$ is such that its curvature locus at the origin is a Roman Steiner surface. Its associated net $q_f=(x^2+yz,z^2-yz,xz+yz)$ is $\mathcal{G}$-equivalent to the germ $(x^2+z^2,z(z-y),z(x+y))$ and using the changes of coordinates in source $\bar{x}=x+y$ and $\bar{y}=z-y$, one can show that $q_f$ is $\mathcal{G}$-equivalent to $(y^2+2z^2,xz,yz)$, which is equivalent to the germ $F_a$ on Table \ref{other types}, whose curvature locus at the origin is a Steiner surface of type 5.
\end{ex}

Example \ref{ex.net} shows that the complete classification of the loci of curvature's topological type
of regular $3$-manifolds in $\mathbb{R}^6$ will require a refinement of the $\mathcal{G}$-classification.
We will discuss this problem, and also the case for corank $1$ $3$-manifolds in $\mathbb{R}^5$, in a forthcoming paper.


\subsection{Curvature locus of a singular 3-manifold in $\mathbb{R}^5$}
For a singular 3-manifold $M$ in $\mathbb{R}^5$ the study of the second fundamental form is analogous, therefore  we can consider the definitions of the first normal space $N_{p}^1M$, affine hull $Aff_p$ and $E_p$ given as previously. In this way, we say that a point $p\in M$ is of type $M_i$, $i=0,1,2,3$ if and only if the second fundamental form has rank $i$ at $p$.

%

We denote by $\mathcal{{S}}_{\overline{M}}$ the residual set of immersions of $\overline{M}$ in $\mathbb{R}^6$ such that every point $p\in \overline{M}$ is of type $\overline{M}_3$ and the curvature locus is a planar region just at isolated points.

From now on we restrict our investigation to the set $\mathcal{S}$ of corank 1 3-manifolds $M$ in $\mathbb{R}^5$ such that at each singular point $p\in M$, there exists a local lifting of $(M,p)$ to $(\overline{M},\overline{p})$ such that $\overline{M}\in \mathcal{{S}}_{\overline{M}} $.

By Remark \ref{obs:cilindro} we have that the subset of unit tangent vectors, $C_q\in T_pM$, is a unity cylinder. In this way, we blow up the sphere at the north and south poles to get a cylinder making
$$x=\cos(\theta),\ y=\sin(\theta)\ \mbox{and}\ z=\dfrac{\cos(\phi)}{\sin(\phi)}, \mbox{such that}\ 0\leq\theta\leq 2\pi,\ 0<\phi<\pi.$$
So we obtain the parametrisation of the curvature locus given in cylindrical coordinates and relate to the curvature locus of a smooth 3-manifold in $\mathbb{R}^6$ given in spherical coordinates, through a local lifting.


The following examples illustrate different possibilities for the curvature locus of a singular 3-manifold in $\mathbb{R}^5$. In each example
below consider the cylindrical coordinates
$$a=\cos(\theta),\ b=\sin(\theta)\ \mbox{and}\ c=\dfrac{\cos(\phi)}{\sin(\phi)},\ 0\leq\theta\leq 2\pi,\ 0<\phi<\pi\ \mbox{and}\ p=(0,0,0,0,0).$$

\begin{ex}\label{ex:1}
 Let the 3-dimensional crosscap in  $\mathbb{R}^5$ given by  $f(x,y,z)=(x,y,xz,yz,z^2)$. By Example \ref{ex:paraboloide}, we have that the curvature locus at  $p$ is a paraboloid given by $\eta(a,b,c)=(2ac, 2bc, 2c^2)$ such that $ a^2+b^2=1$. Taking  cylindrical coordinates,  the paraboloid can be parametrised by
$$\eta(\theta,\phi)=\left(\dfrac{\cos(\theta)\sin(2\phi)}{(\sin(\phi))^2},\dfrac{\sin(\theta)\sin(2\phi)}{(\sin(\phi))^2},\dfrac{1+\cos(2\phi)}{(\sin(\phi))^2}\right).$$

The lifting of $f$ is given by parametrisation  $ \overline{f}(x,y,z)=(x,y,z,xz,yz,z^2)$ of a smooth 3-manifold in  $\mathbb{R}^6$  such that the curvature locus is given by
  $$ \overline{\eta}(\theta,\phi)=\left(\cos(\theta)\sin(2\phi),\sin(\theta)\sin(2\phi),1+\cos(2\phi)\right),$$ that is, an ellipsoid.
\end{ex}

%

\begin{ex}\label{ex:3}
Consider the manifold parametrised by $f(x,y,z)=(x,y,x^2+yz,y^2+xz,z^2+xy)$, so we have that the curvature locus at the origin $p$ is given by $\eta(a,b,c)=(2a^2+c^2,2ac,2bc)$, where $ a^2+b^2=1$. Taking cylindrical coordinates, the curvature locus can be parametrised by $\eta(\theta,\phi)=$
$$\left(\dfrac{2cos(\theta)^2\sin(\phi)^2+\sin(\theta)\sin(2\phi)}{\sin(\phi)^2},\dfrac{2\sin(\theta)^2\sin(\phi)^2+\cos(\theta)\sin(2\phi)}{\sin(\phi)^2},
2\cos(\phi)^2+\sin(2\theta)\right).$$ Figure \ref{fig:cuspidalsw3} illustrates this surface.
We can lift $f$ to obtain the parametrisation $\overline{f}(x,y,z)=(x,y,z,x^2+yz,y^2+xz,z^2+xy)$ of a regular 3-manifold in $\mathbb{R}^6$ such that the
curvature locus at the origin is given by $\overline{\eta}(\theta,\phi)=$
$$\left(2cos(\theta)^2\sin(\phi)^2+\sin(\theta)\sin(2\phi),2\sin(\theta)^2\sin(\phi)^2+\cos(\theta)\sin(2\phi),\sin(\phi)^2(2\cos(\phi)^2+\sin(2\theta))\right),$$ that is, a Roman Steiner surface.
\begin{figure}[H]
\center
\includegraphics[scale=0.5]{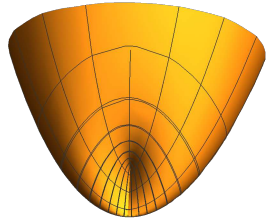}
\caption{Curvature locus of Example \ref{ex:3}}
\label{fig:cuspidalsw3}
\end{figure}
\end{ex}

\begin{ex}\label{ex:4}
Let $f(x,y,z)=(x,y,x^2-\frac{1}{2}yz,y^2-\frac{1}{2}xz, z^2-\frac{1}{2}xy)$, so we have that the curvature locus at the origin $p$ is given by $\eta(a,b,c)=(2a^2-bc,2b^2-ac,2c^2-ab)$ such that $a^2+b^2=1$. Figure \ref{fig:cuspidalsw2} illustrates this curvature locus.
Taking cylindrical coordinates, we can write
$$
\begin{array}{c}
\eta(\theta,\phi)=\Bigg(\dfrac{2(1+\cos(2\theta)\sin(\phi)^2-\sin(\theta)\sin(2\phi))}{(2\sin(\phi))^2},\\
\dfrac{2(1-\cos(2\theta)(\sin(\phi))^2-\cos(\theta)\sin(2\phi))}{2\sin(\phi)^2},\dfrac{2(1+\cos(2\phi)-\sin(2\theta)\sin(\phi)^2)}
{2\sin(\phi)^2}\Bigg).
\end{array}
$$
We can lift $f$ to obtain the parametrisation $\overline{f}(x,y,z)=(x,y,z,x^2-\frac{1}{2}yz,y^2-\frac{1}{2}xz, z^2-\frac{1}{2}xy)$ of a smooth 3-manifold in  $\mathbb{R}^6$ such that the curvature locus at the origin is given by
$$\begin{array}{c}
\overline{\eta}(\theta,\phi)=\Bigg((1+\cos(2\theta)\sin(\phi)^2-\dfrac{\sin(\theta)\sin(2\phi)}{2},\\
(1-\cos(2\theta)\sin(\phi)^2-\dfrac{\cos(\theta)\sin(2\phi)}{2},1+\cos(2\phi)-\dfrac{\sin(2\theta)\sin(\phi)^2}{2}\Bigg),
\end{array}
$$
that is, a Crosscap surface.

%

\begin{figure}[!htb]\hspace{4.5cm}
\begin{minipage}[b]{0.3\linewidth}
\includegraphics[scale=0.4]{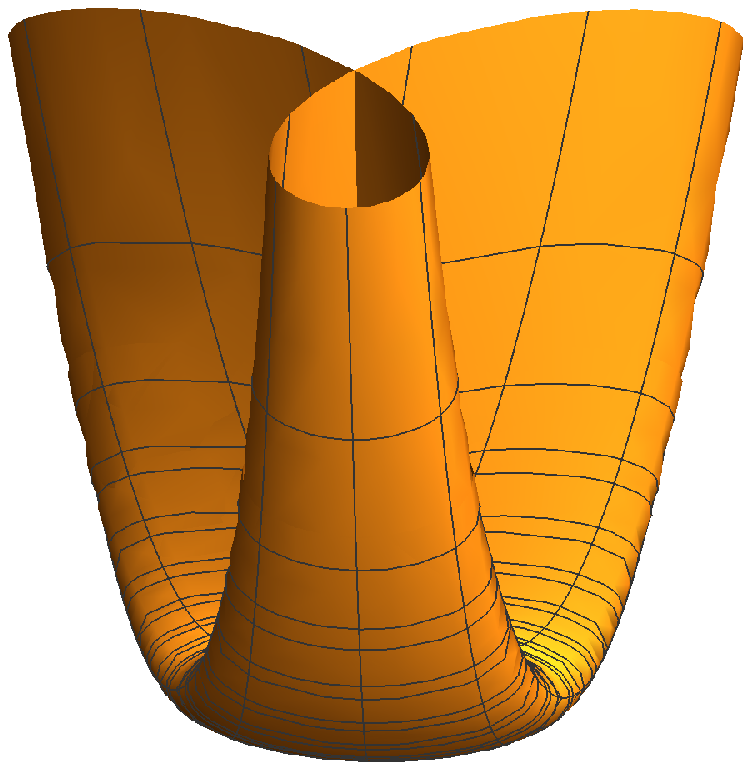}
\end{minipage} \hfill
\begin{minipage}[b]{0.3\linewidth}\hspace{-2cm}
\includegraphics[scale=0.3]{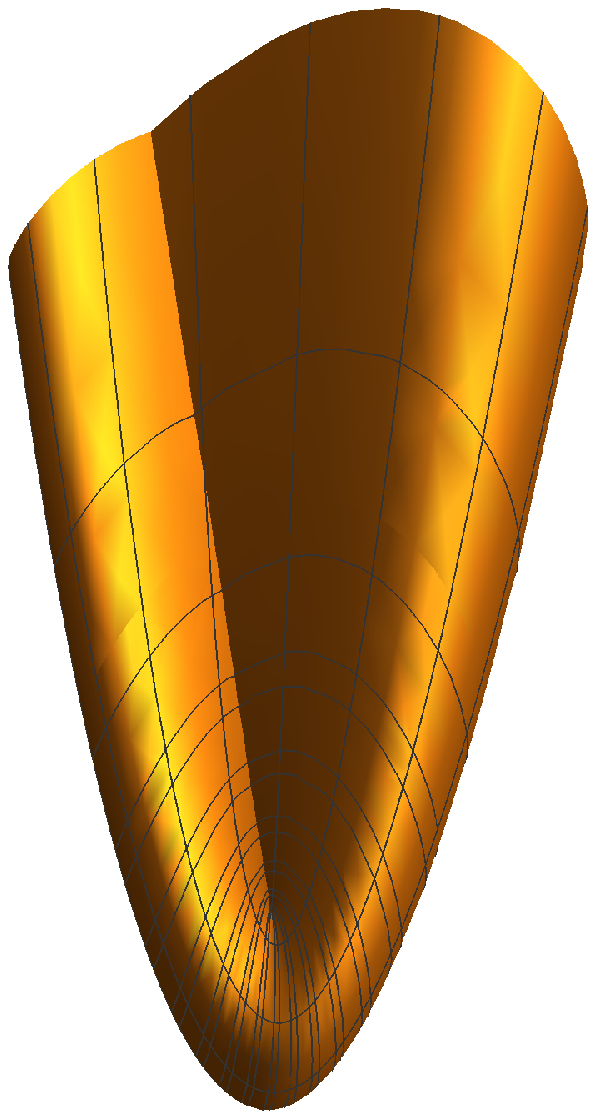}
\end{minipage}
\caption{Curvature locus of Example \ref{ex:4}}
\label{fig:cuspidalsw2}
\end{figure}

\end{ex}

\begin{ex}\label{ex:5}
Consider the manifold parametrised by $f(x,y,z)=(x,y,3z^2,xy+\frac{1}{2}xz,yz)$. The curvature locus at the origin $p$ is given by $\eta(a,b,c)=(6c^2,2ab+ac,2bc)$ such that $ a^2+b^2=1$. Taking cylindrical coordinates, we may write $$\eta(\theta,\phi)=\left(\dfrac{3+3\cos(2\phi)}{\sin(\phi)^2},\dfrac{1}{2}\left(\dfrac{2\sin(2\theta)\sin(\phi)^2+\cos(\theta)\sin(2\phi)}{\sin(\phi)^2}\right)
,\dfrac{\sin(\theta)\sin(2\phi)}{\sin(\phi)^2}\right).$$
Figure \ref{fig:cuspidalsw} illustrates the curvature locus.
Taking the lifting of $f$, we obtain the parametrisation $\overline{f}(x,y,z)=(x,y,z,3z^2,xy+\frac{1}{2}xz,yz)$ of a regular 3-manifold in $\mathbb{R}^6$ whose curvature locus at the origin is given by
$$\overline{\eta}(\theta,\phi)=\left(3+3\cos(2\phi),\dfrac{1}{2}(2\sin(2\theta)\sin(\phi)^2+\cos(\theta)\sin(2\phi)),\sin(\theta)\sin(2\phi)\right),$$ that is, a Steiner surface of type 5.

%

\begin{figure}[!htb]\hspace{4.5cm}
\begin{minipage}[b]{0.3\linewidth}
\includegraphics[scale=0.4]{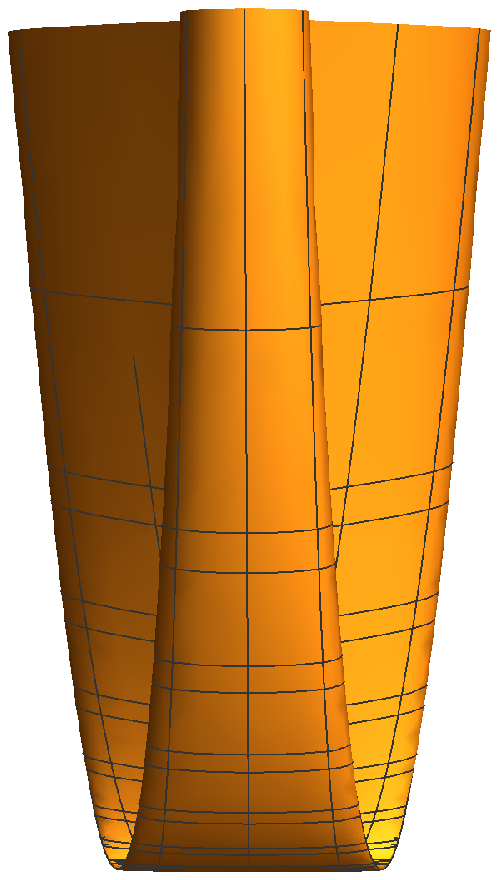}
\end{minipage} \hfill
\begin{minipage}[b]{0.3\linewidth}\hspace{-2cm}
\includegraphics[scale=0.4]{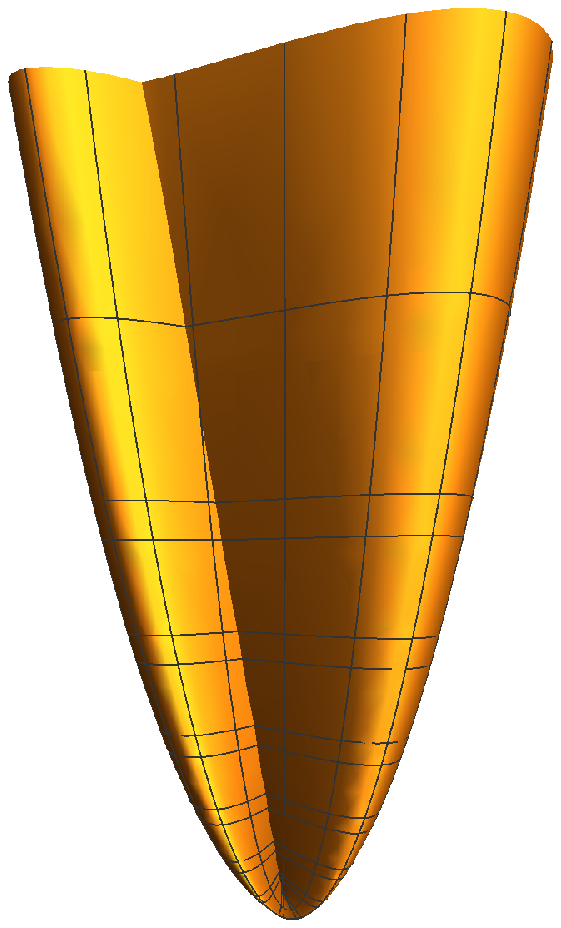}
\end{minipage}
\caption{Curvature locus of Example \ref{ex:5}}
\label{fig:cuspidalsw}
\end{figure}

\end{ex}

\begin{ex}\label{ex:6}
Consider the manifold parametrised by $f(x,y,z)=(x,y,3z^2,x^2+xz+\frac{1}{2}z^2,yz)$. In this case, the curvature locus at $p$, $\Delta_p$, is given by $\eta(a,b,c)=(6c^2,c^2+2a^2+2ac,2bc)$, where $a^2+b^2=1$. Taking cylindrical coordinates, $\Delta_p$ may be parametrised by
$$\eta(\theta,\phi)=\left(\dfrac{3+3\cos(2\phi)}{(\sin(\phi))^2},\dfrac{\cos(2\theta)\sin(\phi)^2+\cos(\theta)\sin(2\phi)+1}{(\sin(\phi))^2},\dfrac{\sin(\theta)\sin(2\phi)}{(\sin(\phi))^2}\right).$$
Figure \ref{fig:cuspidalsw1} shows $\Delta_p$.
Taking the lifting of $f$, we obtain the parametrisation  $\overline{f}(x,y,z)=(x,y,z,3z^2,x^2+xz+\frac{1}{2}z^2,yz)$ of a regular 3-manifold in $\mathbb{R}^6$  such that the curvature locus at the origin is given by
$$\overline{\eta}(\theta,\phi)=\left(3+3\cos(2\phi),\cos(2\theta)\sin(\phi)^2+\cos(\theta)\sin(2\phi)+1,\sin(\theta)\sin(2\phi)\right),$$ that is, a Cross-Cup surface.


\begin{figure}[H]
\center
\includegraphics[scale=0.45]{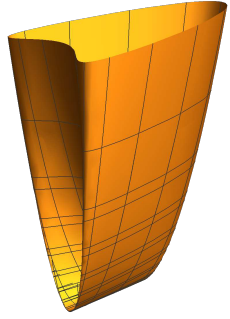}
\caption{Curvature locus of Example \ref{ex:6}}
\label{fig:cuspidalsw1}
\end{figure}
\end{ex}

\section{Appendix: Proof of Theorem \ref{teo:isometria}}\label{appendix}

\begin{proof}
Assume that $j^2f(0)$, $j^2g(0)$ have both $\mathcal{A}^2$-type $(x,y,xz,yz,z^2)$. Then, by Proposition \ref{prop:formanormal}, they can be reduced to the form:
$$
\begin{array}{ll}
j^2f(0)= & (x,y,a_1x^2+a_3y^2+xz+a_6yz,b_1x^2+b_2xy+b_3y^2+b_6yz,\\
 & c_1x^2+c_2xy+c_3y^2+c_4z^2+c_5xz+c_6yz),
\end{array}
$$

$$
\begin{array}{ll}
j^2g(0)= & x,y,\overline{a}_1x^2+\overline{a}_3y^2+xz+\overline{a}_6yz,\overline{b}_1x^2+\overline{b}_2xy+\overline{b}_3y^2+\overline{b}_6yz,\\
& \overline{c}_1x^2+\overline{c}_2xy+\overline{c}_3y^2+\overline{c}_4z^2+\overline{c}_5xz+\overline{c}_6yz),
\end{array}
$$

\noindent with $a_5,\overline{a}_5,c_4,\overline{c}_4>0$. We denote the coordinates in $\mathbb{R}^5$ by $(u,v,w,u',v')$. Then the normal space $N_{p_1}M_1$, $N_{p_2}M_2$ are both equal to the $(w,u',v')$-space and the curvature locus $\Delta_{p_1}(M_1)$, $\Delta_{p_2}(M_2)$  are parametrised respectively by\\

$\eta_1(x,z)=(0,0,a_1x^2+a_3(1-x^2)+2xz+2a_6z\sqrt{1-x^2},b_1x^2+2b_2x\sqrt{1-x^2}+b_3(1-x^2)+2b_6z\sqrt{1-x^2},c_1x^2+2c_2x\sqrt{1-x^2}+c_3(1-x^2)+c_4z^2+2c_5xz+2c_6z\sqrt{1-x^2})$ and \\

$\eta_2(x,z)=(0,0,\overline{a}_1x^2+\overline{a}_3(1-x^2)+2xz+2\overline{a}_6z\sqrt{1-x^2},\overline{b}_1x^2+2\overline{b}_2x\sqrt{1-x^2}+\overline{b}_3(1-x^2)+2\overline{b}_6z\sqrt{1-x^2},
\overline{c}_1x^2+2\overline{c}_2x\sqrt{1-x^2}+\overline{c}_3(1-x^2)+\overline{c}_4z^2+2\overline{c}_5xz+2\overline{c}_6z\sqrt{1-x^2})$.\\

Suppose there is $(\psi,\phi)\in\mathcal{R}^2\times \mathcal{O}(5)$ such that $\phi\circ j^2f(0)\circ\psi=j^2g(0)$. Since they are both homogeneous we can assume that $\psi$ is also a linear map. We denote the matrices of $\psi$ and $\phi$, respectively, by
$$P=\left(
    \begin{array}{ccc}
      d & e & f \\
      g & h & i \\
      j & k & l \\
    \end{array}
  \right),\,\,\,A=\left(
                  \begin{array}{ccccc}
                    a_{11} & a_{12} & a_{13} & a_{14} & a_{15} \\
                     a_{21} & a_{22} & a_{23} & a_{24} & a_{25}  \\
                     a_{31} & a_{32} & a_{33} & a_{34} & a_{35}  \\
                     a_{41} & a_{42} & a_{43} & a_{44} & a_{45} \\
                    a_{51} & a_{52} & a_{53} & a_{54} & a_{55}  \\
                  \end{array}
                \right),
$$
where $\det P\neq0$ and $AA^t=I$.

By comparing coefficients in $\phi\circ j^2f(0)\circ\psi$ and  $j^2g(0)$ we get the following system of 45 equations and 34 variables:

\item $a_{11}d+a_{12}g=1$, $a_{11}e+a_{12}h=0$, $a_{11}f+a_{12}i=0$, $a_{21}d+a_{22}g=0$, $a_{21}e+a_{22}h=1$,

\item $a_{21}f+a_{22}i=0$, $a_{31}d+a_{32}g=0$, $a_{31}e+a_{32}h=0$, $a_{31}f+a_{32}i=0$, $a_{41}d+a_{42}g=0$,

 \item $a_{41}e+a_{42}h=0$, $a_{41}f+a_{42}i=0$, $a_{51}d+a_{52}g=0$, $a_{51}e+a_{52}h=0$, $a_{51}f+a_{52}i=0$,

  \item $a_{15}c_6gj+a_{15}c_2dj+a_{15}c_5dj+a_{15}c_4j^2+a_{13}dj+a_{13}a_6gj+a_{14}b_3g^2+a_{14}b_6gj+a_{15}c_1d^2+a_{15}c_3g^2+a_{13}a_1d^2+a_{13}a_3g^2+a_{14}b_1d^2+a_{14}b_2dg=0,$

  \item $a_{15}c_6hk+a_{15}c_2eh+a_{15}c_5ek+a_{15}c_4k^2+a_{13}ek+a_{13}a_6hk+a_{14}b_3h^2+a_{14}b_6hk+a_{15}c_1e^2+a_{15}c_3h^2+a_{13}a_1e^2+a_{13}a_3h^2+a_{14}b_1e^2a_{14}b_2eh=0,$

  \item $a_{15}c_6il+a_{15}c_2fi+a_{15}c_5fl+a_{15}c_4l^2+a_{13}fl+a_{13}a_6il+a_{14}b_3i^2+a_{14}b_6il+a_{15}c_1f^2+a_{15}c_3i^2+a_{13}a_1f^2+a_{13}a_3i^2+a_{14}b_1f^2+a_{14}b_2fi=0,$

  \item $a_{14}b_6gl+a_{14}b_6ij+a_{15}c_6ij+2a_{15}c_4jl+a_{15}c_5dl+a_{15}c_5fj+a_{15}c_6gl+a_{15}c_1df+a_{15}c_2di+a_{15}c_2fg+2a_{15}c_3gi+a_{13}a_6gl+a_{13}a_6ij+2a_{14}b_1df+
  2a_{14}b_3gi+2a_{13}a_1df+2a_{13}a_3gi+a_{13}dl+a_{13}fj+a_{14}b_2di+a_{14}b_2fg=0,$

    \item $a_{14}b_6hl+a_{14}b_6ik+a_{15}c_6hl+2a_{15}c_4kl+a_{15}c_6ik+a_{15}c_5el+a_{15}c_5fk+2a_{15}c_1ef+a_{15}c_2ei+a_{15}c_2fh+2a_{15}c_3hi+a_{13}a_6hl+a_{13}a_6ik+2a_{14}b_1ef+
  2a_{14}b_3hi+2a_{13}a_1ef+2a_{13}a_3hi+a_{13}el+a_{13}fk+a_{14}b_2ei+a_{14}b_2fh=0,$

    \item $a_{14}b_6hj+2a_{13}a_3gh+2a_{14}b_3gh+a_{14}b_6gk+a_{15}c_6hj+2a_{15}c_4jk+a_{15}c_5dk+a_{15}c_5ej+a_{15}c_6gk+a_{15}c_1de+a_{15}c_2dh+a_{15}c_2eg+a_{13}a_6gk+a_{13}a_6hj+
  2a_{14}b_1de+2a_{13}a_1de+a_{13}dk+a_{13}ej+2a_{15}c_3gh=0+a_{14}b_2dh+a_{14}b_2eg,$

  \item $a_{25}c_6gj+a_{25}c_2dg+a_{25}c_5dj+a_{25}c_4j^2+a_{23}dj+a_{23}a_6gj+a_{24}b_3g^2+a_{24}b_6gj+a_{25}c_1d^2+a_{25}c_3g^2+a_{23}a_1d^2+a_{23}a_3g^2+a_{24}b_1d^2+a_{24}b_2dg=0,$

   \item $a_{25}c_6hk+a_{25}c_2eh+a_{25}c_5ek+a_{23}a_6hk+a_{25}c_4k^2+a_{23}ek+a_{24}b_3h^2+a_{24}b_6hk+a_{25}c_1e^2+a_{25}c_3h^2+a_{23}a_1e^2+a_{23}a_3h^2+a_{24}b_1e^2+a_{24}b_2eh=0,$

   \item $a_{25}c_6il+a_{25}c_5fl+a_{23}a_6il+a_{25}c_2fi+a_{25}c_4l^2+a_{23}fl+a_{24}b_1f^2+a_{24}b_3i^2+a_{24}b_6il+a_{25}c_1f^2+a_{25}c_3i^2+a_{23}a_3i^2+a_{23}a_1f^2+a_{24}b_2fi=0,$

     \item $a_{24}b_6gl+a_{24}b_6ij+a_{25}c_6ij+2a_{25}c_4jl+a_{25}c_5dl+a_{25}c_5fj+a_{25}c_6gl+a_{25}c_1df+a_{25}c_2di+a_{25}c_2fg+2a_{25}c_3gi+a_{23}a_6gl+a_{23}a_6ij+2a_{24}b_1df+
  2a_{24}b_3gi+2a_{23}a_1df+2a_{23}a_3gi+a_{23}dl+a_{23}fj+a_{24}b_2di+a_{24}b_2fg=0,$

    \item $a_{24}b_6hl+a_{24}b_6ik+a_{25}c_6hl+2a_{25}c_4kl+2a_{25}c_6ik+a_{25}c_5el+a_{25}c_5fk+2a_{25}c_1ef+a_{25}c_2ei+a_{25}c_2fh+2a_{25}c_3hi+a_{23}a_6hl+a_{23}a_6ik+2a_{24}b_1ef+
  2a_{24}b_3gi+2a_{23}a_1ef+2a_{23}a_3hi+a_{23}el+a_{23}fk+a_{24}b_2ei+a_{24}b_2fh=0,$

   \item $a_{24}b_6hj+2a_{25}c_3gh+2a_{24}b_3gh+a_{24}b_6gk+2a_{23}a_3gh+a_{25}c_6hj+2a_{25}c_4jk+a_{25}c_5dk+a_{25}c_5ej+a_{25}c_6gk+2a_{25}c_1de+a_{25}c_2dh+a_{25}c_2eg+a_{23}a_6gk+
  2a_{24}b_1de+a_{23}a_6hj+2a_{23}a_1de+a_{23}dk+a_{23}ej+a_{24}b_2dh+a_{24}b_2eg=0,$

   \item $a_{35}c_6gj+a_{35}c_2dg+a_{35}c_5dj+a_{33}a_6jg+a_{33}dj+a_{35}c_1d^2+a_{35}c_3g^2+a_{35}c_4j^2+a_{33}a_3g^2+a_{34}b_1d^2+a_{34}b_3g^2+a_{34}b_6gj+a_{33}a_1d^2
       +a_{34}b_2dg=\overline{a}_1,$

  \item $a_{35}c_6hk+a_{35}c_2eh+a_{35}c_5ek+a_{33}a_6hk+a_{33}ek+a_{35}c_1e^2+a_{35}c_3h^2+a_{35}c_4k^2+a_{33}a_1e^2+a_{34}b_1e^2+a_{34}b_3h^2+a_{34}b_6hk+a_{33}a_3h^2
      +a_{34}b_2eh=\overline{a}_3,$

  \item $a_{35}c_6il+a_{35}c_2fi+a_{35}c_5fl+a_{33}a_6il+a_{33}fl+a_{33}a_6il+a_{35}c_1f^2+a_{35}c_3i^2+a_{35}c_4l^2+a_{33}a_3i^2+a_{33}a_1f^2+a_{34}b_1f^2+a_{34}b_3i^2+a_{34}b_6il
     +a_{34}b_2fi=0,$

 \item $a_{35}c_6ij+a_{35}c_6gl+a_{35}c_5fj+a_{35}c_2fg+2a_{35}c_3gi+2a_{35}c_4jl+a_{35}c_5dl+2a_{34}b_3gi+2a_{35}c_1df+a_{35}c_2di+a_{33}fj+a_{33}a_6gl+a_{33}a_6ij+2a_{34}b_1df+
  a_{34}b_6ij+a_{34}b_6gl+2a_{33}a_1df+2a_{33}a_3gi+a_{33}dl+a_{34}b_2di+a_{34}b_2fg=1,$

  \item $a_{35}c_6ik+a_{35}c_6hl+a_{35}c_5fk+a_{35}c_2fh+2a_{35}c_3hi+2a_{35}c_4kl+a_{35}c_5el+2a_{34}b_1ef+2a_{35}b_3hi+a_{35}c_1ef+a_{35}c_2ei+a_{33}fk+a_{33}a_6hl+a_{33}a_6ik+
  a_{34}hl+a_{34}ik+2a_{33}a_1ef+2a_{33}a_3hi+a_{33}el+a_{34}b_2ei+a_{34}b_2fh=\overline{a}_6,$

\item $a_{35}c_6hj+a_{35}c_6gk+2a_{35}c_4jk+a_{35}c_5dk+a_{35}c_5ej+2a_{35}c_1de+a_{35}c_2dh+a_{35}c_2eg+a_{33}a_6gk+a_{33}a_6hj+2a_{34}b_1de+2a_{35}c_3gh+2a_{33}a_3gh+2a_{34}b_3g^2+
  a_{34}b_6gk+a_{34}b_6hj+2a_{33}a_1de+a_{33}dk+a_{33}ej+a_{34}b_2dh+a_{34}b_2eg=0,$

  \item $a_{45}c_6gj+a_{45}c_2dg+a_{45}c_5dj+a_{43}a_6jg+a_{43}dj+a_{45}c_1d^2+a_{45}c_3g^2+a_{45}c_4j^2+a_{43}a_3g^2+a_{44}b_1d^2+a_{44}b_3g^2+a_{44}b_6gj+a_{43}a_1d^2
       +a_{44}b_2dg=\overline{b}_1,$

 \item $a_{45}c_6hk+a_{45}c_2eh+a_{45}c_5ek+a_{43}a_6jg+a_{43}ek+a_{45}c_1e^2+a_{45}c_3h^2+a_{45}c_4k^2+a_{43}a_3h^2+a_{44}b_1e^2+a_{44}b_3h^2+a_{44}b_6hk+a_{43}a_1e^2
       +a_{44}b_2eh=\overline{b}_3,$

  \item $a_{45}c_6il+a_{45}c_2fi+a_{45}c_5fl+a_{43}a_6il+a_{43}fl+a_{45}c_1f^2+a_{45}c_3i^2+a_{45}c_4l^2+a_{43}a_3i^2+a_{44}b_1f^2+a_{44}b_3i^2+a_{44}b_6il+a_{43}a_1f^2
      +a_{44}b_2fi =0,$

 \item $a_{45}c_6ij+a_{45}c_6gl+a_{45}c_5fj+a_{45}c_2fg+2a_{45}c_3gi+2a_{45}c_4jl+a_{45}c_5dl+2a_{44}b_3gi+2a_{45}c_1df+a_{45}c_2di+a_{43}fj+a_{43}a_6gl+a_{43}a_6ij+2a_{44}b_1df+
  a_{44}b_6ij+a_{44}b_6gl+2a_{43}a_1df+2a_{43}a_3gi+a_{43}dl+a_{44}b_2di+a_{44}b_2fg=0,$

  \item $a_{45}c_6ik+a_{45}c_6hl+a_{45}c_5el+a_{45}c_2fh+2a_{45}c_3hi+2a_{45}c_4kl+a_{45}c_5fk+2a_{44}b_3hi+2a_{45}c_1ef+a_{45}c_2ei+a_{43}el+a_{43}a_6hl+a_{43}a_6ik+2a_{44}b_1ef+
  a_{44}b_6ik+a_{44}b_6hl+2a_{43}a_1ef+2a_{43}a_3hi+a_{43}fk+a_{44}b_2ei+a_{44}b_2fh=\overline{b}_6,$

 \item $a_{45}c_6gk+a_{45}c_6hj+a_{45}c_5ej+a_{45}c_2eg+2a_{45}c_3hg+2a_{45}c_4kj+a_{45}c_5dk+2a_{44}b_3hg+2a_{45}c_1ed+a_{45}c_2dh+a_{43}dk+a_{43}a_6hj+a_{43}a_6gk+2a_{44}b_1ed+
  a_{44}b_6gk+a_{44}b_6hj+2a_{43}a_1ed+2a_{43}a_3hg+a_{43}ej+a_{44}b_2dh+a_{44}b_2eg=\overline{b}_2,$

  \item $a_{55}c_6gj+a_{55}c_2dg+a_{55}c_5dj+a_{53}a_6jg+a_{53}dj+a_{55}c_1d^2+a_{55}c_3g^2+a_{55}c_4j^2+a_{53}a_3g^2+a_{54}b_1d^2+a_{54}b_3g^2+a_{54}b_6gj+a_{53}a_1d^2
       +a_{54}b_2dg =\overline{c}_1,$

 \item $a_{55}c_6hk+a_{55}c_2eh+a_{55}c_5ek+a_{53}a_6hk+a_{53}ek+a_{55}c_1e^2+a_{55}c_3h^2+a_{55}c_4k^2+a_{53}a_3h^2+a_{54}b_1e^2+a_{54}b_3h^2+a_{54}b_6hk+a_{53}a_1e^2
       +a_{54}b_2eh =\overline{c}_3,$

 \item $a_{55}c_6il+a_{55}c_2fi+a_{55}c_5fl+a_{53}a_6il+a_{53}fl+a_{55}c_1f^2+a_{55}c_3i^2+a_{55}c_4l^2+a_{53}a_3i^2+a_{54}b_1f^2+a_{54}b_3i^2+a_{54}b_6il+a_{53}a_1f^2
       +a_{54}b_2fi =\overline{c}_4,$

  \item $a_{55}c_6ij+a_{55}c_6gl+a_{55}c_5fj+a_{55}c_2fg+2a_{55}c_3gi+2a_{55}c_4jl+a_{55}c_5dl+2a_{54}b_3gi+2a_{55}c_1df+a_{55}c_2di+a_{53}fj+a_{53}a_6gl+a_{53}a_6ij+2a_{54}b_1df+
  a_{54}b_6ij+a_{54}b_6gl+2a_{53}a_1df+2a_{53}a_3gi+a_{53}dl+a_{54}b_2di+a_{54}b_2fg=\overline{c}_5,$

  \item $a_{55}c_6ik+a_{55}c_6hl+a_{55}c_5el+a_{55}c_2fh+2a_{55}c_3hi+2a_{55}c_4kl+a_{55}c_5fk+2a_{54}b_3hi+2a_{55}c_1ef+a_{55}c_2ei+a_{53}el+a_{53}a_6hl+a_{53}a_6ik+2a_{54}b_1ef+
  a_{54}b_6ik+a_{54}b_6hl+2a_{53}a_1ef+2a_{53}a_3hi+a_{53}fk+a_{54}b_2ei+a_{54}b_2fh=\overline{c}_6,$

  \item $a_{55}c_6gk+a_{55}c_6hj+a_{55}c_5ej+a_{55}c_2eg+2a_{55}c_3hg+2a_{55}c_4kj+a_{55}c_5dk+2a_{54}b_3hg+2a_{55}c_1ed+a_{55}c_2dh+a_{53}dk+a_{53}a_6hj+a_{53}a_6gk+2a_{54}b_1ed+
  a_{54}b_6gk+a_{54}b_6hj+2a_{53}a_1ed+2a_{53}a_3hg+a_{53}ej+a_{54}b_2dh+a_{54}b_2eg=\overline{c}_2.$ \\

  From the analysis of this system using software MAPLE, we deduce that
  $d,h,l\neq0$, $e=f=g=i=j=k=0$, $a_{11},a_{22},a_{33},a_{44}\neq0$,  $a_{55}>0$, $a_{12},a_{13},a_{14},a_{15},a_{21},a_{23}, a_{24},$ $ a_{25},a_{31},a_{32},a_{35},a_{41}, a_{42},a_{43},a_{45},a_{51},a_{52}=0$. Now using the fact that $AA^t=I$, we have $a_{11}=\pm1$, $a_{22}=\pm1$, $a_{33}=\pm1$, $a_{44}\pm1$, $a_{55}=1$, $a_{34},a_{53},a_{54}=0$. Hence the system reduces to:

  \item $a_{11}d=1$, $a_{22}h=1$, $a_{33}a_3h^2=\overline{a}_3$, $a_{33}dl=1$,  $a_{33}a_6dl=\overline{a}_6$,  $a_{44}b_6hl=\overline{b}_6$, $a_{44}b_2dh=\overline{b}_2$,

  \item   $a_{44}b_1d^2=\overline{b}_1$, $a_{44}b_3h^2=\overline{b}_3$, $a_{55}c_1d^2=\overline{c}_1$, $a_{55}c_2dh=\overline{c}_2$, $a_{55}c_3h^2=\overline{c}_3$, $a_{55}c_4l^2=\overline{c}_4$,   $a_{55}c_5dl=\overline{c}_5$,  $a_{55}c_6hl=\overline{c}_6$.
  Thus, there are only sixteen possible solutions to the system:

  \begin{enumerate}
    \item $d=h=l=1$, $a_{11}, a_{22},a_{33},a_{44},a_{55}=1$ and $(\overline{a}_1,\overline{a}_3,\overline{a}_6,\overline{b}_1,\overline{b}_2,\overline{b}_3,\overline{b}_6,\overline{c}_1,\overline{c}_2,\overline{c}_3,
        \overline{c}_4,\overline{c}_5,$ $\overline{c}_6)= (a_1,{a}_3,a_6, {b}_1,b_2,{b}_3,b_6,{c}_1,{c}_2,{c}_3,{c}_4,{c}_5,{c}_6);$

    \item $d=-1$, $h=l=1$, $a_{11}=a_{33}=-1$, $a_{22}=a_{44}=a_{55}=1$ and $(\overline{a}_1,\overline{a}_3,\overline{a}_6,\overline{b}_1,\overline{b}_2,\overline{b}_6,\overline{b}_3,$ $\overline{c}_1, \overline{c}_2,\overline{c}_3,\overline{c}_4,\overline{c}_5,\overline{c}_6)=
        (-a_1,-{a}_3,-a_6, {b}_1,-b_2,{b}_3,b_6,{c}_1,-{c}_2,{c}_3,{c}_4,-{c}_5,{c}_6);$

    \item $d=1$, $h=l=-1$,  $a_{11}=a_{44}=a_{55}=1$, $a_{22}=a_{33}=-1$ and $(\overline{a}_1,\overline{a}_3,\overline{a}_6,\overline{b}_1,\overline{b}_2,\overline{b}_3,\overline{b}_6,$ $\overline{c}_1, \overline{c}_2,\overline{c}_3,\overline{c}_4,\overline{c}_5,\overline{c}_6)=
        (-a_1,-{a}_3,-a_6, {b}_1,-b_2,{b}_3,b_6{c}_1,-{c}_2,{c}_3,{c}_4,-{c}_5,{c}_6);$

    \item $d=h=l=-1$, $a_{11}=a_{22}=-1$, $a_{33}=a_{44}=a_{55}=1$ and $(\overline{a}_1,\overline{a}_3,\overline{a}_6,\overline{b}_1,\overline{b}_2,\overline{b}_3,\overline{b}_6,\overline{c}_1,$ $ \overline{c}_2, \overline{c}_3,\overline{c}_4,\overline{c}_5,\overline{c}_6)=
        (a_1,{a}_3,a_6, {b}_1,b_2,{b}_3,b_6,{c}_1,{c}_2,{c}_3,{c}_4,{c}_5,{c}_6);$

    \item $d=h=-1$, $l=1$, $a_{11}=a_{22}=a_{33}=a_{44}=-1$, $a_{55}=1$ and $(\overline{a}_1,\overline{a}_3,\overline{a}_6,\overline{b}_1,\overline{b}_2,\overline{b}_3,\overline{b}_6,$ $\overline{c}_1, \overline{c}_2, \overline{c}_3,\overline{c}_4,\overline{c}_5,\overline{c}_6)=
        (-a_1,-{a}_3,a_6, -{b}_1,-b_2,-{b}_3,b_6,{c}_1,{c}_2,{c}_3,{c}_4,-{c}_5,-{c}_6);$

    \item  $d=h=1$, $l=-1$, $a_{11}=a_{22}=a_{55}=1$, $a_{33}=a_{44}=-1$ and $(\overline{a}_1,\overline{a}_3,\overline{a}_6,\overline{b}_1,\overline{b}_2,\overline{b}_3,\overline{b}_6,$ $\overline{c}_1, \overline{c}_2, \overline{c}_3,\overline{c}_4,\overline{c}_5,\overline{c}_6)=
        (-a_1,-{a}_3,a_6, -{b}_1,-b_2,-{b}_3,b_6,{c}_1,{c}_2,{c}_3,{c}_4,-{c}_5,-{c}_6)$;

          \item  $d=h=1$, $l=-1$, $a_{11}=a_{22}=a_{55}=1$, $a_{33}=a_{44}=1$ and $(\overline{a}_1,\overline{a}_3,\overline{a}_6,\overline{b}_1,\overline{b}_2,\overline{b}_3,\overline{b}_6,$ $\overline{c}_1,$ $ \overline{c}_2, \overline{c}_3,\overline{c}_4,\overline{c}_5,\overline{c}_6)=
        (-a_1,-{a}_3,a_6, {b}_1,b_2,{b}_3,-b_6{c}_1,{c}_2,{c}_3,{c}_4,-{c}_5,-{c}_6)$;

            \item  $d=h=l=1$, $a_{11}=a_{22}=a_{33}=a_{55}=1$, $a_{44}=-1$ and $(\overline{a}_1,\overline{a}_3,\overline{a}_6,\overline{b}_1,\overline{b}_2,\overline{b}_3,\overline{b}_6,$ $\overline{c}_1,$ $ \overline{c}_2, \overline{c}_3,\overline{c}_4,\overline{c}_5,\overline{c}_6)=
        (a_1,{a}_3,a_6, -{b}_1,-b_2,-{b}_3,-b_6,{c}_1,{c}_2,{c}_3,{c}_4,{c}_5,{c}_6)$;

            \item  $d=l=-1$, $h=1$, $a_{11}=-1$, $a_{22}=a_{33}=a_{44}=a_{55}=1$ and $(\overline{a}_1,\overline{a}_3,\overline{a}_6,\overline{b}_1,\overline{b}_2,\overline{b}_3,\overline{b}_6,$ $\overline{c}_1,$ $ \overline{c}_2, \overline{c}_3,\overline{c}_4,\overline{c}_5,\overline{c}_6)=
        (a_1,{a}_3,-a_6, {b}_1,-b_2,{b}_3,-b_6,{c}_1,-{c}_2,{c}_3,{c}_4,{c}_5,-{c}_6)$;

        \item  $d=h=1$, $l=-1$, $a_{11}=a_{22}=a_{44}=a_{55}=1$, $a_{33}=-1$ and $(\overline{a}_1,\overline{a}_3,\overline{a}_6,\overline{b}_1,\overline{b}_2,\overline{b}_3,\overline{b}_6,$ $\overline{c}_1,$ $ \overline{c}_2, \overline{c}_3,\overline{c}_4,\overline{c}_5,\overline{c}_6)=
        (-a_1,-{a}_3,a_6, {b}_1,b_2,{b}_3,-b_6,{c}_1,{c}_2,{c}_3,{c}_4,-{c}_5,-{c}_6)$;

        \item  $d=h=-1$, $l=1$, $a_{11}=a_{22}=a_{33}=-1$, $a_{44}=a_{55}=1$ and $(\overline{a}_1,\overline{a}_3,\overline{a}_6,\overline{b}_1,\overline{b}_2,\overline{b}_3,\overline{b}_6,$ $\overline{c}_1,$ $ \overline{c}_2, \overline{c}_3,\overline{c}_4,\overline{c}_5,\overline{c}_6)=
        (-a_1,-{a}_3,a_6, {b}_1,b_2,{b}_3,-b_6,{c}_1,{c}_2,{c}_3,{c}_4,-{c}_5,-{c}_6)$;

        \item  $d=h=l=-1$, $a_{11}=a_{22}=a_{44}=-1$, $a_{33}=a_{55}=1$ and $(\overline{a}_1,\overline{a}_3,\overline{a}_6,\overline{b}_1,\overline{b}_2,\overline{b}_3,\overline{b}_6,$ $\overline{c}_1,$ $ \overline{c}_2, \overline{c}_3,\overline{c}_4,\overline{c}_5,\overline{c}_6)=
        (a_1,{a}_3,a_6, -{b}_1,-b_2,-{b}_3,-b_6,{c}_1,{c}_2,{c}_3,{c}_4,{c}_5,{c}_6)$;

        \item  $d=1$, $h=l=-1$, $a_{11}=a_{55}=1$, $a_{22}=a_{33}=a_{44}=-1$ and $(\overline{a}_1,\overline{a}_3,\overline{a}_6,\overline{b}_1,\overline{b}_2,\overline{b}_3,\overline{b}_6,$ $\overline{c}_1,$ $ \overline{c}_2, \overline{c}_3,\overline{c}_4,\overline{c}_5,\overline{c}_6)=
        (-a_1,-{a}_3,-a_6, -{b}_1,b_2,-{b}_3,-b_6,{c}_1,-{c}_2,{c}_3,{c}_4,-{c}_5,{c}_6)$;

        \item  $d=-1$, $h=l=1$, $a_{11}=a_{33}=a_{44}=-1$, $a_{22}=a_{55}=1$ and $(\overline{a}_1,\overline{a}_3,\overline{a}_6,\overline{b}_1,\overline{b}_2,\overline{b}_3,\overline{b}_6,$ $\overline{c}_1,$ $ \overline{c}_2, \overline{c}_3,\overline{c}_4,\overline{c}_5,\overline{c}_6)=
        (-a_1,-{a}_3,-a_6, -{b}_1,b_2,-{b}_3,-b_6,{c}_1,-{c}_2,{c}_3,{c}_4,-{c}_5,{c}_6)$;

    \item  $d=l=-1$, $h=1$, $a_{11}=a_{44}=-1$, $a_{22}=a_{33}=a_{55}=1$ and $(\overline{a}_1,\overline{a}_3,\overline{a}_6,\overline{b}_1,\overline{b}_2,\overline{b}_3,\overline{b}_6,$ $\overline{c}_1,$ $ \overline{c}_2, \overline{c}_3,\overline{c}_4,\overline{c}_5,\overline{c}_6)=
        (a_1,{a}_3,-a_6, -{b}_1,b_2,-{b}_3,b_6,{c}_1,-{c}_2,{c}_3,{c}_4,{c}_5,-{c}_6)$;

         \item  $d=-1$, $h=l=1$, $a_{11}=a_{33}=-1$, $a_{22}=a_{44}=a_{55}=1$ and $(\overline{a}_1,\overline{a}_3,\overline{a}_6,\overline{b}_1,\overline{b}_2,\overline{b}_3,\overline{b}_6,$ $\overline{c}_1,$ $ \overline{c}_2, \overline{c}_3,\overline{c}_4,\overline{c}_5,\overline{c}_6)=
        (-a_1,-{a}_3,-a_6, -{b}_1,-b_2,{b}_3,b_6,{c}_1,-{c}_2,{c}_3,{c}_4,-{c}_5,{c}_6)$.

  \end{enumerate}
  In any case, we conclude that $\phi (\Delta_{p_1}(M_1))=\Delta_{p_2}(M_2)$.

  For the converse, suppose that there is a linear isometry $\phi$ in the $(w,u',v')$-space such that $\phi (\Delta_{p_1}(M_1))=\Delta_{p_2}(M_2)$. We denote by $\widetilde{A}$ the matrix of $\phi$. By comparing coefficients in the equations of the curvature locus and using similar arguments, we find there are only two possibilities:

  \begin{enumerate}
    \item $\widetilde{A}=\left(
                           \begin{array}{ccc}
                             1 & 0 & 0 \\
                             0 & 1 & 0 \\
                             0 & 0 & 1 \\
                           \end{array}
                         \right)$ and  $(\overline{a}_1,\overline{a}_3,\overline{a}_6,\overline{b}_1,
                         \overline{b}_2,\overline{b}_3,\overline{b}_6,\overline{c}_1,\overline{c}_2,
                         \overline{c}_3,\overline{c}_4,\overline{c}_5,\overline{c}_6)= (a_1,{a}_3,{a}_5,a_6, {b}_1,$ $b_2,{b}_3,b_6,{c}_1,{c}_2,{c}_3,{c}_4,{c}_5,{c}_6)$;

                           \item $\widetilde{A}=\left(
                           \begin{array}{ccc}
                             1 & 0 & 0 \\
                             0 & -1 & 0 \\
                             0 & 0 & 1 \\
                           \end{array}
                         \right)$ and  $(\overline{a}_1,\overline{a}_3,\overline{a}_6,\overline{b}_1,\overline{b}_2,\overline{b}_3,\overline{b}_6,\overline{c}_1,\overline{c}_2,
                         \overline{c}_3,\overline{c}_4,\overline{c}_5,\overline{c}_6)= (a_1,{a}_3,a_6, {b}_1,$ $b_2,{b}_3,-b_6,{c}_1,{c}_2,{c}_3,{c}_4,{c}_5,{c}_6)$.

  \end{enumerate}
Hence, we extend $\phi$ to a linear isometry of $\mathbb{R}^5$ in the obvious way, so that $\phi\circ j^2f(0)=j^2g(0)$.
The remaining cases where the $2$-jets have type $(x,y,z^2,xz,0)$, $(x,y,xz,yz,0)$, $(x,y,z^2,0,0)$, $(x,y,xz,0,0)$ or $(x,y,0,0,0)$ are treated in a similar way.
\end{proof}

{\small

\par\noindent
P. Benedini Riul, Departamento de Matem\'{a}tica, Universidade Federal de S\~{a}o Carlos, Caixa Postal 676, CEP 13560-905, S\~{a}o Carlos-SP , Brazil
\par\noindent
e-mail: \texttt{pedro.benedini.riul@gmail.com}
\\
\par\noindent
M. A. S. Ruas and A. de J. Sacramento,  Departamento de Matem\'{a}tica,
ICMC Universidade de S\~{a}o Paulo, Campus de S\~{a}o Carlos, Caixa Postal 668, CEP 13560-970, S\~{a}o Carlos-SP, Brazil
\par\noindent
e-mails: \texttt{maasruas@icmc.usp.br, anddyunesp@yahoo.com.br}


}

\end{document}